\documentclass[11pt]{article}
\usepackage[utf8]{inputenc}
\usepackage[T1]{fontenc}
\usepackage{lmodern}
\usepackage[normalem]{ulem}
\usepackage[makeroom]{cancel}

\hsize \textwidth
\advance \hsize by -\marginparwidth
\evensidemargin \oddsidemargin

\usepackage{amsthm, amssymb, amsmath, amsfonts, mathrsfs}

\usepackage{mathtools}
\usepackage[ocgcolorlinks, colorlinks=true, pdfstartview=FitV, linkcolor=blue, citecolor=blue, urlcolor=blue, pagebackref=false]{hyperref}

\usepackage{microtype}

\usepackage{color}
\usepackage{accents}

\usepackage{bbm}

\definecolor{labelkey}{gray}{.8}
\definecolor{refkey}{gray}{.8}

\definecolor{darkred}{rgb}{0.9,0.1,0.1}
\definecolor{darkgreen}{rgb}{0,0.5,0}

\newtheorem{theorem}{Theorem}[section]
\newtheorem{lemma}[theorem]{Lemma}

\newtheorem{corollary}[theorem]{Corollary}
\newtheorem{proposition}[theorem]{Proposition}

\theoremstyle{remark}
\newtheorem{remark}[theorem]{Remark}

\renewenvironment{proof}[1][Proof]{ {\itshape \noindent {#1.}} }{$\Box$
\medskip}

\numberwithin{equation}{section}
\newcommand{\R}{\mathbb{R}}

\newcommand{\Z}{\mathbb{Z}}
\newcommand{\N}{\mathbb{N}}
\newcommand{\Pb}{\mathbb{P}}

\newcommand{\E}{\mathbb{E}}

\newcommand{\F}{\mathcal{F}}
\newcommand{\cF}{\mathscr{F}}

\newcommand{\cG}{\mathscr{G}}
\newcommand{\M}{\mathbf{M}}
\newcommand{\A}{\mathcal{A}}

\newcommand{\C}{\mathcal{C}}

\newcommand{\I}{\mathcal{I}}
\newcommand{\J}{\mathcal{J}}

\newcommand{\cU}{\mathscr{U}}

\newcommand{\eps}{\varepsilon}
\def\les{\lesssim}

\newcommand{\Var}{\mathrm{Var}}

\newcommand{\la}{\langle}
\newcommand{\ra}{\rangle}

\newcommand{\X}{\mathbf{X}}
\newcommand{\cX}{\mathcal{X}}

\newcommand{\Y}{\mathbf{Y}}

\newcommand{\bfa}{\boldsymbol a}

\newcommand{\bfpi}{\boldsymbol\pi}

\newcommand{\what}{\widehat}

\newcommand{\bfx}{\boldsymbol x}
\newcommand{\bfs}{\boldsymbol s}

\def\blue{\textcolor{blue}}

\newcommand{\farc}{\frac}

\begin{document}

 % \blue{The} \corO{ Edwards-Wilkinson limits for random heat equations}}
\title{The Edwards-Wilkinson limit
of the  random heat equation in dimensions three and higher}

\author{Yu Gu, Lenya Ryzhik, Ofer Zeitouni}
\date{}

\maketitle

\begin{abstract}
We consider  the heat equation with a 
\textit{multiplicative}
Gaussian potential in dimensions $d\geq 3$. We show that the renormalized solution
converges to the solution of a deterministic diffusion equation %and identify the effective diffusivity in terms of the microscopic dynamics
with an effective diffusivity.
We also prove that the
renormalized
large scale random fluctuations
%after a suitable renormalization, 
are
described by the Edwards-Wilkinson model,  that is,
the stochastic heat equation (SHE) with \textit{additive}
white noise, with an effective variance.
%The variance of the white noise in It turns out that both the diffusivity 
%\corO{of the limiting heat equation} and the variance of the \corO{additive}
%noise are homogenized.
%in the limit.
\end{abstract}

\section{Introduction  }

We consider  the  solutions to the heat equation with a smooth Gaussian 
random potential:
\begin{equation}\label{e.maineq}
\partial_t u=\frac12\Delta u+\lambda V(t,x)u,  \    \   x\in\R^d, d\geq 3.
\end{equation}
Here, $\lambda>0$ is a constant, and the random potential $V(t,x)$ 
is a mean-zero Gaussian field
that we assume to be of the form 
\[
V(t,x)=\int_{\R^{d+1}}\phi(t-s)\psi(x-y)dW(s,y),
\]
%\sout{is constructed from 
%\corO{a}}
where  $dW(s,y)$ is a
space-time white noise built on a probability 
space $(\Sigma,\F,\Pb)$.
We assume that
the non-negative functions $\phi,\psi\in \C_c^\infty$, that
$\phi$ is supported on $[0,1]$, and that
$\psi$ is even and supported on~$\{x: |x|\leq 1/2\}$.  The covariance function of $V$ is
\begin{equation}\label{sep1414}
R(t,x)=\E[V(0,0)V(t,x)]=\int_{\R}\phi(t{+}s)\phi(s)ds\int_{\R^d}\psi(x{+}y)\psi(y)dy.
\end{equation}
Here, $\E$ denotes the expectation on $\Sigma$. The above assumptions on the correlation function $R(t,x)$ are made
mostly to simplify the notation, and the only essential technical assumptions are that $R(t,x)$ is compactly supported in $t$ and is rapidly decaying
in $x$.

As we are interested in the large scale and long time asymptotics of $u(t,x)$,
we consider the rescaled function  
%$\eps\ll1$, we define 
\[
u_\eps(t,x):=u(\frac{t}{\eps^2},\frac{x}{\eps}),
\]
with $\eps\ll 1$. The function $u_\eps$ 
satisfies %\corO{the re-scaled equation}
\begin{equation}\label{e.maineq1}
\partial_t u_\eps=\frac12\Delta u_\eps+ \frac{\lambda}{\eps^2}V(\frac{t}{\eps^2},\frac{x}{\eps})u_\eps.%  \   \  u_\eps(0,x)=u_0(x)\in \C_b(\R^d).
\end{equation}
We assume that
the initial condition $u_\eps(0,x)=u_0(x)\in \C_b(\R^d)$. 
Throughout the paper, we stay in the weak disorder regime and assume 
that $\lambda\in(0,\lambda_0)$, 
%\sout{ is fixed, }
with  a small but fixed constant $\lambda_0$ only 
depending on $d$, $\phi$ and $\psi$. 
Our main result is as follows.
\begin{theorem}\label{t.mainth}
There exist $c_1,c_2$ depending on $\lambda,\phi$, and $\psi$ such 
that for any $t>0$ and $g\in \C_c^\infty(\R^d)$, we have
\begin{equation}\label{e.conmean}
\int_{\R^d}u_\eps(t,x)\exp\Big\{-\farc{c_1t}{\eps^2}-c_2\Big\}g(x)dx\to 
\int_{\R^d}\bar{u}(t,x)g(x)dx,~~\hbox{ as $\eps\to0$,}
\end{equation}
in probability, and 
\begin{equation}\label{e.mainthCLT}
\frac{1}{\eps^{{d}/{2}-1}}\int_{\R^d}(u_\eps(t,x)-\E[u_\eps(t,x)])\exp\Big\{-\farc{c_1t}{\eps^2}-c_2\Big\}
g(x)dx\Rightarrow \int_{\R^d} \cU(t,x)g(x)dx
\end{equation}
in distribution. Here, $\bar{u}$ is the solution of the effective
heat equation 
 \begin{equation}\label{e.heathom}
\partial_t \bar{u}=\frac12\nabla\cdot \bfa_{\mathrm{eff}}\nabla\bar{u},  \  \  \bar{u}(0,x)=u_0(x),
\end{equation}
with the 
effective diffusion matrix $\bfa_{\mathrm{eff}}\in\R^{d\times d}_{\mathrm{sym}}$ 
  defined in \eqref{e.defaeff} below,
and $\cU$ solves the additive stochastic heat equation
\begin{equation}\label{e.ewshe}
\partial_t \cU=\frac12\nabla\cdot \bfa_{\mathrm{eff}}\nabla \cU+\lambda\nu_{\mathrm{eff}}\bar{u}\dot{W}, \   \  \cU(0,x)=0,
\end{equation}
with the
effective variance $\nu_{\mathrm{eff}}^2>0$ defined in \eqref{e.defnu} below.
\end{theorem}
The renormalization constants $c_1$ and $c_2$ are identified in 
  (\ref{sep2206}) below.

\subsection{ {Background} and related problems}

The study of singular stochastic PDEs has witnessed important progress in recent years, with different approaches developed to make sense of 
%nonlinear 
equations which are genuinely ill-posed due to the lack 
of regularity and the need to make sense of the  multiplication
of distributions \cite{gubinelli2015paracontrolled,hairer2013solving,hairer2014theory,kupiainen2016renormalization,otto2016quasilinear}. 
The existing works typically prove that the solution of the
equation with the 
mollified white noise, after a suitable renormalization, 
converges to some limit 
that is independent of the way in which
the noise is mollified. 
%\blue{\sout{The focus is mostly on the subcritical case, 
%where only finitely many terms need to be proved to converge after suitable 
%renormalizations, and the solution is shown to depend continuously 
%on those terms, hence is defined pathwise with respect to the 
%realization of the random noise.}}
 
Here, we consider a slightly different situation: 
the rescaled  random field in \eqref{e.maineq1} 
is not a mollification of the white
noise, and does not directly converge 
to the white noise in~$d\geq 3$ as $\eps\to0 $. We rather have, formally,
\[
\farc{1}{\eps^2}V\big(\farc{t}{\eps^2},\farc{x}{\eps}\big)\sim \eps^{d/2-1}\nu_0\dot W(t,x),
\]
with 
\begin{equation}\label{sep2402}
\nu_0^2=\int_{\R^{d+1}} R(s,y)dsdy.
\end{equation}
Hence, one could think that the noise in (\ref{e.maineq1}) is small and would not produce a non-trivial effect on the solutions, so that the
limit would be simply the unperturbed heat equation. This is problematic -- if we formally replace the random potential 
in (\ref{e.maineq1}) by $\eps^{d/2-1}\dot W(t,x)$, we obtain the multiplicative stochastic heat equation. Giving a meaning to its 
solutions in~$d\ge 3$ brings about the aforementioned question of making sense of multiplying two distributions~$u$ and~$\dot W$.
Hence, the issue of the limit is much more delicate.
%As we discuss below, this is, indeed, "almost" the case in some related regimes. 
%However,  here, this simple guess  would be incorrect. One reason is that trying to formally approximate that perceived trivial limit by  
%replacing the potential in~\eqref{e.maineq1} 
%by~$\eps^{d/2-1}\dot W(t,x)$  results in the multiplicative stochastic heat equation with a space-time white noise potential. 
%Giving a meaning to its solutions in~$d\ge 3$ brings about
%the aforementioned question of making sense of multiplying two distributions~$u$ and~$\dot W$. This need to renormalize 
%the solutions of the hypothetical "next-order correction" destroys the naive leading order guess, and the correct limit can not 
%be the unmodified deterministic heat equation.  
%\sout{Here , we consider the linear heat equation with a multiplicative random potential, and develop a probabilistic approach to analyze a somewhat different scaling regime, where the random field in \eqref{e.maineq1} does not directly converge to the white noise in $d\geq 3$. }}
Theorem~\ref{t.mainth} 
shows that even though the random potential in (\ref{e.maineq1}) formally converges to zero, it still affects the solutions in a non-trivial way:  
(i)~on the level of the
law of large numbers, the solution of (\ref{e.maineq1}) converges to a solution of the 
deterministic diffusion equation  \eqref{e.heathom}, 
with an effective diffusivity that is 
modified by the presence of the noise, and (ii) on the level of the
central limit theorem, 
the random fluctuations, after a rescaling, fall into the Edwards-Wilkinson universality class in~$d\geq 3$, as in \eqref{e.ewshe},
with an effective (and not a ``naive-guess'' $\nu_0$) variance.  
We stress that   both the diffusion matrix and the variance 
of the noise are homogenized in the limit.

 We mention two related problems. The weak coupling regime   
analyzed in~\cite{gu2015homogenization}
concerns the situation when the potential in (\ref{e.maineq})
is asymptotically small:
\begin{equation}\label{sep1302}
\partial_t u=\frac12\Delta u+\eps V(t,x)u,  \    \   x\in\R^d, d\geq 3.
\end{equation}
It was shown that no renormalization is required:
the diffusively rescaled solution 
\[
  u_\eps(t,x)=u(t/\eps^2,x/\eps)e^{-V_{\mathrm{eff}}t}
\]
converges in probability to the solution of the diffusion equation
\begin{equation}\label{sep1402}
\partial_t \bar{u}=\frac12\Delta\bar{u},  \  \  \bar{u}(0,x)=u_0(x),
\end{equation}
with an un-modified diffusivity. 
The effective potential $V_{\mathrm{eff}}$
is explicit:
\[
V_{\mathrm{eff}}=\int_0^\infty \E_B[R(t,B_t)]dt.
\]
As far as {fluctuations are  concerned}, using a simpler version of what is done in the present paper, one can show 
that for any $t>0$ and $g\in \C_c^\infty(\R^d)$ we have, as $\eps\to0$:
\begin{equation}
\frac{1}{\eps^{{d}/{2}}}\int_{\R^d}(u_\eps(t,x)-\E[u_\eps(t,x)])e^{-V_{\mathrm{eff}}t}g(x)dx\Rightarrow \int_{\R^d} \cU(t,x)g(x)dx
\end{equation}
in distribution. Here, $\cU$ solves the stochastic heat equation with %an 
additive space-time white noise 
\begin{equation}\label{sep1304}
\partial_t \cU=\frac12\Delta \cU+\nu_0
%\xcancel{\hat{R}^{1/2}(0,0)}
\bar{u}\dot{W}, \   \  \cU(0,x)=0.
\end{equation}
%satisfies 
%\begin{equation}\label{e.maineqwk}
%\partial_t u_\eps=\frac12\Delta u_\eps+\frac{\lambda}{\eps}V(\frac{t}{\eps^2},\frac{x}{\eps})u_\eps,  \   \  u_\eps(0,x)=u_0(x),
%\end{equation}
%The asymptotic limit of $u_\eps$ was analyzed in \cite{gu2015homogenization}, where a homogenization of the large random potential is proved for a general class of mixing random fields. }
%Define 
%\[
%V_{\mathrm{eff}}=\lambda^2\int_0^\infty \E_B[R(t,B_t)]dt,
%\]
%then for any $g\in \C_c^\infty(\R^d)$, it was shown that
%\begin{equation}
%\int_{\R^d}u_\eps(t,x)e^{-V_{\mathrm{eff}}t}g(x)dx\to \int_{\R^d}\bar{u}(t,x)g(x)dx
%\end{equation}
%in probability, with $\bar{u}$ solving the heat equation 
% 
% Let 
%\[
%\hat{R}(0,0)=\int_{\R^{d+1}}R(t,x)dtdx=\left(\int_{\R}\phi(t)dt\int_{\R^d} \psi(x)dx\right)^2.
%\]
%Under our assumption on $V$, the same argument shows 
%\begin{theorem}
%For any $t>0$ and $g\in \C_c^\infty(\R^d)$, as $\eps\to0$,
%\begin{equation}
%\frac{1}{\eps^{\frac{d}{2}}}\int_{\R^d}[u_\eps(t,x)-\E[u_\eps(t,x)])e^{-V_{\mathrm{eff}}t}g(x)dx\Rightarrow \int_{\R^d} \cU(t,x)g(x)dx
%\end{equation}
%in distribution, where $\cU$ solves the SHE with an additive space-time white noise 
%\begin{equation}
%\partial_t \cU=\frac12\Delta \cU+\lambda \hat{R}^\frac12(0,0)\bar{u}\dot{W}, \   \  \cU(0,x)=0.
%\end{equation}
%\end{theorem}
%
%Due to the weak random potential, we do not observe a homogenized diffusivity. Furthermore, since
Note that neither the diffusivity nor the variance of the noise in (\ref{sep1304}) are homogenized in the weak coupling regime.
Indeed, equations (\ref{sep1402}) and (\ref{sep1304}) are precisely the ``naive guesses'' for the leading order equation and its
approximation that fail in our case, when the potential is not weak -- it has no pre-factor $\eps$ in (\ref{e.maineq})
unlike in (\ref{sep1302}).

The case when $V$ is white in time but not in space
was considered in \cite{mukherjee2016weak}: %. Let $\phi(t)=\delta(t)$, so
\[
V(t,x)=\dot{W}_\psi(t,x)=
\int\psi(x-y) dW(t,y).
\]
%and 
%\[
%R(t,x)=\E[V(0,0)V(t,x)]=\delta(t)R_\psi(x).
%\]
Equation \eqref{e.maineq} is interpreted in \cite{mukherjee2016weak} in the It\^o sense:
\begin{equation}\label{sep1702}
\partial_t u=\frac12\Delta u+\lambda \dot{W}_\psi(t,x)u,   \  \  x\in\R^d, d\geq 3.
\end{equation}
It was shown in \cite[Theorem 2.1]{mukherjee2016weak} 
that there exists $\lambda_1>0$ so that if
$\lambda\in(0,\lambda_1)$, %and the initial condition for (\ref{sep1702}) is slowly varying: $u(0,x)=u_0(\eps x)$, 
then the rescaled solution $u_\eps(t,x)=u(t/\eps^2,x/\eps)$
%\sout{with~$u_\eps(0,x)=u_0(x)$,} 
satisfies 
\[
\int_{\R^d} u_\eps(t,x)g(x)dx\to\int_{\R^d} \bar{u}(t,x)g(x)dx
\]
in probability for any 
$g\in\C_c^\infty(\R^d)$. Here, $\bar{u}$ solves the heat equation 
\[
\partial_t \bar{u}=\frac12\Delta \bar{u}, \   \  \bar{u}(0,x)=u_0(x),
\]
with an un-modified diffusivity.
The same approach as in the present paper  gives in that case
\begin{theorem}\label{t.white}
  There exists   {$\lambda_1=\lambda_1(\psi)$}
    so that for all
$0<\lambda<\lambda_1$ we have 
%as $\eps\to0$:
\begin{equation}
\frac{1}{\eps^{{d}/{2}-1}}\int_{\R^d}(u_\eps(t,x)-\E[u_\eps(t,x)])g(x)dx\Rightarrow \int_{\R^d} \cU(t,x)g(x)dx,
\end{equation}
 {as $\eps\to0$},
with $\cU$ solving 
\begin{equation}
\partial_t \cU=\frac12\Delta \cU+\lambda \nu_{\mathrm{eff}} \bar{u}\dot{W}, \   \  \cU(0,x)=0,
\end{equation}
and 
\[
\nu_{\mathrm{eff}}^2=\int_{\R^d} R_\psi(x)\E_B\Big[\exp\Big\{\frac12\lambda^2\int_0^\infty R_\psi(x+B_s)ds\Big\}\Big]dx.
\]
\end{theorem}
In this case, only the variance of the noise is homogenized 
but not the diffusivity.
%\sout{ are discussed in the appendix, where the random potential has either a smaller size or a correlation that is white-in-time. }
Thus, both  
these regimes also lead to an Edwards-Wilkinson limit, 
with an un-modified diffusivity, 
and with either a ``naive-guess'' noise variance (the weak coupling case), 
or a homogenized noise variance (in the white
in time case),
%\sout{\corO{with either a vanishing additive noise (the weak coupling) or identity
%diffusivity (the white in time case)}} 
whereas 
\eqref{e.maineq1} leads to both homogenized diffusivity and variance.

We  
mention  the very recent paper \cite{mukherjee2017central} 
that considers essentially the same {setup} 
as in the present paper. The main result of \cite{mukherjee2017central}  
implies (\ref{e.conmean}) except that
the convergence is established 
for the averages and not in probability, 
and the renormalization  in the exponent 
is less explicit than in (\ref{e.conmean}).

 In dimensions $d=1,2$, similar problems have been discussed
in the literature. 
For the random PDE \eqref{e.maineq1}, with $\lambda=\lambda(\eps)\to0$ 
chosen appropriately, and after a possible renormalization, 
the solution~$u_\eps$  converges to the solution to the 
stochastic heat equation with %a 
multiplicative space-time white noise in $d=1$ 
\cite{chandra2017moment,gu2018another,hairer2015multiplicative,hairer2015wong}, and a Gaussian field in $d=2$ within the 
weak-disorder regime~\cite{caravenna2015universality,feng2016rescaled}. 
For random polymers and interacting particle systems, 
the partition function or the height function plays the role of 
the solution to certain ``PDE'', and their convergences to the 
SHE/KPZ equation have been proved in $d=1$ e.g. 
in \cite{alberts2014intermediate,amir2011probability,bertini1997stochastic}.

 {We} 
comment briefly on the strategy of the proof.
The Feynman-Kac representation expresses the solution to the random PDE in the form of a partition function of a directed polymer in 
a random environment, and the appearance of the effective diffusivity in the limit can be interpreted as the convergence of a
diffusively rescaled polymer path converging
to a Brownian motion in $d\geq 3$, see 
the results in~\cite{betz2005central,gubinelli2006gibbs,mukherjee2017central} 
for the annealed continuous setting and 
\cite{bolthausen1989note,imbrie1988diffusion} for
the quenched discrete setting. 
By a construction similar to \cite{mukherjee2017central}, 
we utilize the finite range in time correlation of~$V(t,x)$ to 
decompose the polymer path into length-one increments and 
establish a Markovian dynamics in the space of path increments. 
 The
latter Markov chain satisfies the Doeblin condition,  greatly 
simplifying the analysis. 
The proof of the Edwards-Wilkinson limit for the
fluctuations relies on the Clark-Ocone formula which 
expresses the random fluctuation in terms of a stochastic integral, and 
the fact that $u_\eps(t,x)$ essentially only depends 
on  {$dW(s,x)$} locally around $s=t/\eps^2$. 

It may be possible to apply a PDE approach,  such as using 
the correctors 
in the standard homogenization theory, to identity the limit and 
prove the convergence. However, the particular scaling considered here 
requires the construction
of infinitely many correctors.  Controlling these correctors becomes
increasingly more difficult  as their order increases. 
Therefore, we find the probabilistic methods more convenient to use here. 

Finally, we comment on our assumption of $\lambda\ll1$. We choose the disorder to be weak enough so that the $L^2(\Omega)$ norm of the (rescaled) solution is uniformly bounded in $\eps$, or equivalently, the corresponding random polymer is in the $L^2$ regime \cite[Chapter 3]{comets2}. As we increase $\lambda$ to enter the strong disorder regime, localization type of behaviors of the random PDE/polymer will appear which is beyond the scope of the paper. It is worth mentioning that there are different notions of the critical temperature which separates the weak and strong disorder regimes, see \cite[page 27, Theorem 2.4]{comets2} and \cite[page 34, Proposition 3.1]{comets2}. For our interest in the fluctuations of the random PDE, the critical $\lambda_c$ is the one beyond which the effective variance $\nu_{\mathrm{eff}}^2$ becomes infinite. We also mention that in the context of weak disorder polymer, a pointwise version of \eqref{e.mainthCLT} was obtained in \cite{comets1}.

\subsection{Connections to the KPZ equation}

The recent work \cite{magnen2017diffusive}, which employs completely
different methods,
is closely related to ours. It
considers the KPZ equation,  related to \eqref{e.maineq1}  by
a Cole-Hopf transformation. 
The setup and result are close but not exactly the same as here and we
discuss below the connection.

Starting from \eqref{e.maineq}, applying the centered
Cole-Hopf transformation 
\[
h(t,x)=\lambda^{-1}\log u(t,x)-c_0t,
\]
one obtains
\begin{equation}\label{sep1310}
\partial_t h=\frac12\Delta h+\frac12\lambda |\nabla h|^2+V(t,x)-c_0,  \   \  x\in\R^d, d\geq 3,
\end{equation}
with a constant $c_0$. Define
\begin{equation}
h_\eps(t,x):=\eps^{-{d}/{2}+1}h(\frac{t}{\eps^2},\frac{x}{\eps}),
\end{equation}
which satisfies
\begin{equation}
\partial_t h_\eps=\frac12\Delta h_\eps+\frac12\lambda\eps^{{d}/{2}-1} |\nabla h_\eps|^2+
\eps^{-{d}/{2}-1}\left(V(\frac{t}{\eps^2},\frac{x}{\eps})-c_0\right).
\end{equation}
The rescaled random potential $\eps^{-{d}/{2}-1}V({t}/{\eps^2},{x}/{\eps})$ converges to the space-time white noise, 
while the nonlinear term formally disappears as $\eps\to0$ in $d\geq3$. 
The authors in \cite[Theorem 0.1]{magnen2017diffusive} show %claim 
that if the initial condition $h_0(x)$ for 
the un-scaled KPZ equation (\ref{sep1310}) is rapidly decaying, then
for~$\lambda$ sufficiently small, 
the Edwards-Wilkinson model shows up in the limit: 
%\sout{with renormalized coefficients:}
\begin{equation}\label{e.conkpz}
h_\eps(t,x)-\E[h_\eps(t,x)]\to \mathscr{H}(t,x),
\end{equation}
in the sense of convergence of the corresponding
multipoint correlation functions. 
Here, $\mathscr{H}$ is the   solution to
\begin{equation}\label{eq-wrong}
\partial_t \mathscr{H}=\frac12D_{\mathrm{eff}}\Delta \mathscr{H}+\mu_{\mathrm{eff}}\dot{W}
\end{equation}
{with zero initial conditions},
for some $D_{\mathrm{eff}},\mu_{\mathrm{eff}}>0$.
One difference from our setting is that we consider the initial conditions for the 
un-scaled stochastic heat equation (\ref{e.maineq}) that vary on a macroscopic scale:~$u(0,x)=u_0(\eps x)$.
%and not a fixed initial condition 
%of the form
%\begin{equation}\label{sep1312}
%u(0,x)=1+u_0(x),
%\end{equation}
%with a rapidly decaying $u_0(x)$, as
%in~\cite{magnen2017diffusive}. This would lead to the initial condition 
%\begin{equation}\label{sep1314}
%u_\eps(0,x)=1+u_0(x/\eps)
%\end{equation}
%for \eqref{e.maineq1}.
 {Disregarding}
  this difference, we 
try to interpret the convergence in~\eqref{e.conkpz} on the level of 
the stochastic heat equation, using the relation 
\[
u_\eps(t,x)=\exp\left(\lambda \eps^{{d}/{2}-1}h_\eps(t,x)+\frac{\lambda c_0t}{\eps^2}\right).
\] 
%If we replace the initial condition (\ref{sep1314}) by its natural weak limit $u_\eps(0,x)\equiv 1$, 
As proved in Theorem~\ref{t.mainth},
\begin{equation}\label{e.conshekpz}
\frac{1}{\eps^{{d}/{2}-1}} \int\limits_{\R^d} \left(e^{\lambda \eps^{{d}/{2}-1}h_\eps(t,x)}-
\E[e^{\lambda \eps^{{d}/{2}-1}h_\eps(t,x)}]\right)e^{{\lambda c_0t}/{\eps^2}}e^{-{c_1t}/{\eps^2}-c_2}g(x)
dx\Rightarrow \int\limits_{\R^d} \cU(t,x)g(x)dx.
\end{equation}
If we use the approximation 
\begin{equation}\label{e.approx}
e^{\lambda \eps^{{d}/{2}-1}h_\eps(t,x)}\approx 1+\lambda \eps^{{d}/{2}-1}h_\eps(t,x),
\end{equation}
and choose $\lambda c_0=c_1$, then \eqref{e.conkpz} follows from \eqref{e.conshekpz}. {Our current methods however do no provide 
  sufficiently strong error bounds in \eqref{e.approx} to justify the 
approximation. }
%in \eqref{e.approx}. }
{
\begin{remark}
One possible way to check that the effective constants 
in \eqref{eq-wrong} match the ones in \eqref{e.ewshe} is to 
expand them in terms of the coupling constant $\lambda$ and 
compare the coefficients; we do not carry out this comparison here, and note
that in any case, to get high order 
coefficients seems difficult in our setting. 
\end{remark}
}
%\corO{However, the initial conditions in these two results are not the same: 
%whereas in  \eqref{eq-wrong} one takes the \textit{stationary} solution,
%this is definitely not the case in \eqref{e.ewshe}. \blue{The initial condition in~(\ref{e.conshekpz}) would correspond to $\mathscr{H}(t,0)=0$.}
%We are puzzled by this discrepancy.}

\medskip
{\bf Organization of the paper.} The paper is organized as follows. 
In Section~\ref{sec:prelim} we introduce a tilted Brownian motion
and  use the Clark-Ocone formula to establish 
 {in Lemma~\ref{l.stochrepre}} 
a 
representation for the fluctuation as a stochastic integral,
%in Lemma~\ref{l.stochrepre} 
and obtain 
 {in Lemma~\ref{l.varre}}
an expression for its variance. 
In Section~\ref{sec:thm-proof} we prove Theorem~\ref{t.mainth}.
Assuming the main technical result, Proposition~\ref{p.varcon}, 
we show that the fluctuation does depend only on the ``recent past'' of the noise,
and use this to prove the central limit theorem for the fluctuations. 
The proof of Proposition~\ref{p.varcon} presented in Section~\ref{s.pp}
relies on the properties of a Markov chain on the space of path increments that is constructed in Section~\ref{s.chain}. Finally, some
technical results are proved in the Appendix.

\subsection*{Acknowledgments}  We thank Herbert Spohn for bringing 
\cite{magnen2017diffusive} to our attention and motivating this study,
and Chiranjib Mukherjee for making \cite{mukherjee2017central} available to us and 
for useful discussions.
Y.G. was partially supported by the NSF through DMS-1613301 and by the Center for Nonlinear Analysis of CMU, L.R. 
was partially supported by the NSF grant DMS-1613603 and ONR grant N00014-17-1-2145, and
O.Z. was partially supported by a Poincare visiting professorship at Stanford,
by an Israel Science Foundation grant, and
by the ERC advanced grant LogCorFields. 
We would like to thank the anonymous referees 
for helpful comments and suggestions.

\section{Preliminaries: a stochastic integral and variance representation}\label{sec:prelim}

%\section{Proof of \corO{Theorem \ref{t.mainth}}}

The goal in this section is to express 
the deviation of the solution of (\ref{e.maineq}) from its mean
in terms of a stochastic integral   given by
the Clark-Ocone formula, 
and present a convenient formula for its second moment. 
Let~$B$ be a standard Brownian motion starting from the origin 
that is independent 
from the random potential~$V$, and let
$\E_B$ denote the expectation with respect to $B$.  
We define the renormalization constant
\begin{equation}\label{e.renormalization}
\zeta_t:=\log\E_B\Big[\exp\Big\{\frac{\lambda^2}{2}\int_{[0,t]^2}R(s-u,B_s-B_u)dsdu\Big\}\Big],
\end{equation}
and denote by
$\what{\E}_{B,t}$ the expectation 
with respect to a tilted Brownian path on~$[0,t]$: 
for any integrable random variable $f(B)$ depending on~$B=\{B_s: s\geq 0\}$, set
\begin{equation}\label{sep1416}
\what{\E}_{B,t}[f(B)]:=\E_B\Big[f(B) \exp\Big\{\frac{\lambda^2}{2}\int_{[0,t]^2}R(s-u,B_s-B_u)dsdu-\zeta_t\Big\}\Big].
\end{equation}
For two independent tilted Brownian motions $B^1,B^2$ on $[0,t]$, we write
\[
\what{\E}_{B,t}[f(B^1,B^2)]=\E_{B}\Big[f(B^1,B^2) \prod_{i=1}^2 
\exp\Big\{\frac12\lambda^2\int_{[0,t]^2}R(s-u,B_s^i-B_u^i)dsdu- \zeta_t\Big\}\Big].
\]

For $t>0,x\in\R^d$ and every realization of the Brownian motion, we define
    \begin{equation}\label{e.defPhi}
  \Phi_{t,x,B}(s,y):=\int_0^t \phi(t-r-s)\psi(x+B_r-y)dr,
  \end{equation}
and the square-integrable martingale
\begin{equation}
M_{t,x,B}(r):=\int_{-\infty}^r \int_{\R^d}\Phi_{t,x,B}(s,y)dW(s,y),
\end{equation}
with   quadratic variation
\begin{equation}
\la M_{t,x,B}\ra_r=\int_{-\infty}^r \int_{\R^d} |\Phi_{t,x,B}(s,y)|^2dsdy.
\end{equation}
Since $\phi$ is supported on $[0,1]$, $\Phi_{t,x,B}(s,y)\neq 0$ only when $s\in[-1,t]$.

The following lemma expresses the  random 
fluctuations of $u(t,x)$ in terms of a  stochastic integral.
\begin{lemma}\label{l.stochrepre}
Let $u(t,x)$ be a solution to (\ref{e.maineq}), then
for any $t>0$ and $x\in\R^d$, we have 
\begin{equation}\label{sep1406}
(u(t,x)-\E[u(t,x)])e^{-\zeta_t}=\lambda\int\limits_{-1}^t\int\limits_{\R^d} \what{\E}_{B,t}\Big[u(0,x+B_t)\Phi_{t,x,B}(r,y)
\exp\Big\{\lambda M_{t,x,B}(r)-\frac{\lambda^2}2\la M_{t,x,B}\ra_r\Big\} \Big] dW(r,y).
\end{equation}
\end{lemma}
\begin{proof}
Since $\phi(s)=0$ for $s<0$,   $u(t,x)$ 
is adapted to the filtration generated by {$dW$}
up to $t$, denoted by $\F_t$. By the Clark-Ocone formula, we have
\[
u(t,x)-\E[u(t,x)]=\int_{-\infty}^t \E[D_{r,y}u(t,x)|\F_r] dW(r,y).
\]
Here, $D_{r,y}$ denotes the Malliavin derivative.  
As the function $\phi(s)$ is supported in $[0,1]$,
the random potential $V(t,x)$ for $t>0$ depends only on $\dot W(r,y)$ 
for $r>-1$, and so does $u(t,x)$ for $t>0$. 
Therefore, the Malliavin derivative vanishes for $r<-1$, and we have
\begin{equation}\label{sep1412}
u(t,x)-\E[u(t,x)]=\int_{-1}^t \E[D_{r,y}u(t,x)|\F_r] dW(r,y).
\end{equation}
To compute the Malliavin derivative in (\ref{sep1412}), we note that 
by the Feynman-Kac formula, the solution can be written as
\[
u(t,x)=\E_B\Big[u(0,x+B_t)\exp\Big\{\lambda\int_0^t V(t-s,x+B_s)ds\Big\}\Big].
\]
Rewriting the exponent above as 
\[
\begin{aligned}
\int_0^tV(t-s,x+B_s)ds  =&\int_0^t\left( \int_{\R^{d+1}}\phi(t-s-s')\psi(x+B_s-y')dW(s',y')\right) ds\\
=&\int_{\R^{d+1}}\Phi_{t,x,B}(s',y')dW(s',y'),
\end{aligned}
\]
we see that
the Malliavin derivative is given by
\[
\begin{aligned}
D_{r,y}u(t,x)  =\lambda\E_B\Big[u(0,x+B_t)\Phi_{t,x,B}(r,y)\exp\Big\{\lambda\int_{\R^{d+1}}\Phi_{t,x,B}(s',y')dW(s',y')\Big\}\Big ],
\end{aligned}
\]
so  that 
\begin{equation}\label{sep1420}
\E[D_{r,y}u(t,x) |\F_r]
=\lambda\E_B\Big(u(0,x+B_t)\Phi_{t,x,B}(r,y)\E\Big[\exp\Big\{\lambda\int_{
  \R^{d+1}}\Phi_{t,x,B}(s',y')dW(s',y')\Big\}
\Big|\F_r\Big]\Big).
\end{equation}
For the conditional expectation in the right side, we write 
\[
\int_{\R^{d+1}} \Phi_{t,x,B}(s',y')dW(s',y')=\left(\int_{-\infty}^r+\int_r^\infty\right)\int_{\R^d}\Phi_{t,x,B}(s',y') dW(s',y'),
\]
which gives
\begin{equation}\label{sep1418}
\begin{aligned}
\E\Big[\exp\Big\{\lambda\int_{\R^{d+1}}\Phi_{t,x,B}(s',y')dW(s',y')\Big\}\Big|\F_r\Big]
=&\exp\Big\{\lambda M_{t,x,B}(r)-\frac{\lambda^2}{2}\la M_{t,x,B}\ra_r\Big\} \\
\times&  \exp\Big\{\frac{\lambda^2}{2}\int_{\R^{d+1}}|\Phi_{t,x,B}(s',y')|^2ds'dy'\Big\}.
\end{aligned}
\end{equation}
With the help of the definition (\ref{e.defPhi}) of $\Phi_{t,x,B}$,  together with expression (\ref{sep1414}) for $R(t,x)$ and the fact that 
the function~$\psi$ is even,
the last integral in (\ref{sep1418}) can be written as
\begin{equation}\label{sep1422}
\int_{\R^{d+1}}|\Phi_{t,x,B}(s',y')|^2ds'dy'=\int_{[0,t]^2}R(s-u,B_s-B_u)dsdu.
\end{equation}
Finally, using (\ref{sep1420}), \eqref{sep1418} and (\ref{sep1422}), as well as 
the definition (\ref{sep1416}) of the tilted measure $\what\E_{B,t}$, in~(\ref{sep1412}),
completes the proof of (\ref{sep1406}).
\end{proof}
{
\begin{remark}
The Clark-Ocone formula is useful for separating the mean and the random fluctuation of regular random variables. For example, it has been used in the study of Brownian local time in \cite{hu2009ecp,hu2008integral}.% In the study of SPDEs, the use of Malliavin calculus includes establishing the smoothness of the density of the solution \cite{hu2011aop}.
\end{remark}
}
\subsubsection*{An expression for the variance}

We now use Lemma~\ref{l.stochrepre} for the re-scaled solution
$u_\eps(t,x)=u(t/\eps^2,x/\eps)$, with $u_\eps(0,x)=u_0(x)$.
%Since $u_\eps(t,x)=u(t/\eps^2,x/\eps)$ with $u_\eps(0,x)=u_0(x)$, 
For any test function $g\in\C_c^\infty(\R^d)$, we have
\begin{equation}\label{e.store}
\begin{aligned}
\int_{\R^d}(u_\eps(t,x)-\E[u_\eps(t,x)])e^{-\zeta_{t/\eps^2}}g(x)dx
=\lambda\int_{-1}^{{t}/{\eps^2}}\int_{\R^d} Z^\eps_t(r,y)dW(r,y),
\end{aligned}
\end{equation}
with 
\begin{equation}\label{e.defZ}
Z^\eps_t(r,y):=\int_{\R^d} g(x) \what{\E}_{B,{t}/{\eps^2}}
\Big[u_0(x+\eps B_{{t}/{\eps^2}})\Phi^\eps_{t,x,B}(r,y)
\exp\Big\{\lambda M^\eps_{t,x,B}(r)-\farc{\lambda^2}2\la M^\eps_{t,x,B}\ra_r\Big\} \Big] dx,
\end{equation}
where 
\[
\Phi_{t,x,B}^\eps:=\Phi_{{t}/{\eps^2},{x}/{\eps},B},  \    \  M_{t,x,B}^\eps:=M_{{t}/{\eps^2},{x}/{\eps},B}.
\]
Thus, the proof of the fluctuation convergence \eqref{e.mainthCLT} in Theorem~\ref{t.mainth}
reduces to the analysis of the stochastic integral 
\begin{equation}\label{e.stochainte}
\frac{1}{\eps^{{d}/{2}-1}}\int_{\blue -1}^{{t}/{\eps^2}}\int_{\R^d} Z^\eps_t(r,y)dW(r,y),
\end{equation}
provided that we can replace $\zeta_{t/\eps^2}\mapsto c_1t/\eps^2+c_2$ as $\eps\to0$.

We express the variance of the stochastic integral 
in \eqref{e.stochainte} in a more  explicit form. 
First, we need to introduce some notation. We define 
\begin{equation}\label{e.defco}
R_\psi(x)=\int_{\R^d} \psi(x-y)\psi(y)dy,  \   \  R_\phi(t_1,t_2)=\int_0^\infty \phi(s-t_1)\phi(s-t_2)ds.
\end{equation}
Since $\psi$ is supported on $\{x: |x|\leq 1/2\}$ and $\phi$ on $[0,1]$, 
we know that
$R_\psi$ is supported on $\{x:|x|\leq 1\}$ and $R_\phi(t_1,t_2)=0$ if $t_1<-1$ or $t_2<-1$. In addition, $R_\phi(t_1,t_2)=0$ if $|t_1-t_2|\geq 1$.

From now on, we fix $t>0$. Given two continuous paths $B^1,B^2\in \C([0,{t}/{\eps^2}])$, 
 we set
\[
\Delta B^i_{u,v}=B^i_v-B^i_u.
\]
For 
$x_1,x_2,y\in\R^d, s_1,s_2\in[0,1], r\in[0,t]$ and $-1<M_1,M_2\leq r/\eps^2$, we define
%with \corO{$\Delta B^i_{u,v}=B^i_v-B^i_u$,}
\begin{equation}\label{e.defJ1}
\begin{aligned}
\mathcal{\I}_\eps&=\I_\eps(x_1,x_2,y,s_1,s_2,r)= \prod_{i=1}^2 g(\eps x_i+y-\eps B_{\frac{t-r}{\eps^2}-s_i}^i)
u_0(\eps x_i+y+\eps \Delta B^i_{\frac{t-r}{\eps^2}-s_i,
\frac{t}{\eps^2}}),
\end{aligned}
\end{equation}
and
\begin{equation}\label{e.defJ2}
\begin{aligned}
&\mathcal{J}_\eps(M_1,M_2)=\mathcal{J}_\eps(M_1,M_2,x_1,x_2,s_1,s_2,r)\\
&=\lambda^2\int_{-1}^{M_1}\int_{-1}^{M_2}R_\phi(u_1,u_2)R_\psi(x_1-x_2+
  \Delta B^1_{\frac{t-r}{\eps^2}-s_1,\frac{t-r}{\eps^2}+u_1}
  -\Delta B^2_{ \frac{t-r}{\eps^2}-s_2,\frac{t-r}{\eps^2}+u_2})du_1du_2.
\end{aligned}
\end{equation}
To simplify the notation, we write $\I_\eps$ and $\mathcal{J}_\eps(M_1,M_2)$ and keep their dependence on $B^i,x_i,y,s_i,r$ implicit. 
\begin{lemma}\label{l.varre}
For any $-1\leq t_1<t_2\leq t-\eps^2$, we have, with $d\bar s=ds_1ds_2$ and $d\bar x=dx_1dx_2$:
\begin{equation}\label{e.varz}
\begin{aligned}
&\frac{1}{\eps^{d-2}}\E\left[\int_{{t_1}/{\eps^2}}^{{t_2}/{\eps^2}} \int_{\R^d} |Z^\eps_t(r,y)|^2 dydr\right]
= \int_{t_1}^{t_2}\int_{\R^{3d}}\int_{[0,1]^2}\what{\E}_{B,{t}/{\eps^2}}
\Big[\I_\eps e^{\J_\eps(\frac{r}{\eps^2},\frac{r}{\eps^2})}\Big] \prod_{i=1}^2 \phi(s_i)\psi(x_i)\  d\bar s d\bar xdy dr.
\end{aligned}
\end{equation}
\end{lemma}

\begin{proof}
The proof is a straightforward calculation with multiple changes of variables. We first write 
\[
|Z^\eps_t(r,y)|^2=\what{\E}_{B,{t}/{\eps^2}}\int_{\R^{2d}}\prod_{i=1}^2 g(x_i)
u_0(x_i+\eps B_{{t}/{\eps^2}}^i)\Phi^\eps_{t,x_i,B^i}(r,y)
\exp\Big\{\lambda M^\eps_{t,x_i,B^i}(r)-\frac12\lambda^2\la M^\eps_{t,x_i,B^i}\ra_r\Big\}d\bar x,
\]
where we recall that $\what{\E}_{B,t/\eps^2}$ is the expectation with respect to the tilted measure defined in \eqref{sep1416}. Taking the expectation $\E$ above, for each
$B^1,B^2$ fixed we have
\[
\E\left[\prod_{i=1}^2e^{\lambda M^\eps_{t,x_i,B^i}(r)-\frac12\lambda^2\la M^\eps_{t,x_i,B^i}\ra_r}\right]=e^{\lambda^2\la M^\eps_{t,x_1,B^1},M^\eps_{t,x_2,B^2}\ra_r},
\]
with
\[
\la M^\eps_{t,x_1,B^1},M^\eps_{t,x_2,B^2}\ra_r=\int_{-\infty}^r \int_{\R^d}\Phi_{t,x_1,B^1}^\eps(s',z)\Phi_{t,x_2,B^2}^\eps(s',z)dzds'.
\]
Next, we write 
\begin{equation}\label{sep1514}
\Phi^\eps_{t,x_1,B^1}(r,y)\Phi^\eps_{t,x_2,B^2}(r,y)=\int_{[0,{t}/{\eps^2}]^2}
\prod_{i=1}^2\phi(\frac{t}{\eps^2}-s_i-r)\psi(\frac{x_i}{\eps}+B_{s_i}^i-y)ds_1ds_2.
\end{equation}
We consider the integral in $x,y$ and change variables $x_i\mapsto \eps x_i-\eps B_{s_i}^i+\eps y$, $y\mapsto  y/\eps$ to obtain
\begin{equation}\label{sep1516}
\begin{aligned}
&\E\left[  \int_{\R^d} |Z^\eps_t(r,y)|^2 dy\right]
=\what{\E}_{B,{t}/{\eps^2}}\int_{[0,{t}/{\eps^2}]^2}
\int_{\R^{3d}}\prod_{i=1}^2 g(x_i)u_0(x_i+\eps B_{{t}/{\eps^2}}^i)
\psi(\frac{x_i}{\eps}+B^i_{s_i}-y) \phi(\frac{t}{\eps^2}-s_i-r)\\
& \times 
\exp\Big\{\lambda^2\int_{-\infty}^r \int_{\R^d}\Phi_{t,x_1,B^1}^\eps(s',z)\Phi_{t,x_2,B^2}^\eps(s',z)dzds'\Big\} d{\bar x}dy 
d{\bar s}\\
&=\eps^d\what{\E}_{B,{t}/{\eps^2}}\int_{[0,{t}/{\eps^2}]^2}\int_{\R^{3d}}
\prod_{i=1}^2 g(\eps x_i+y-\eps B_{s_i}^i)u_0(\eps x_i+y+\eps B_{{t}/{\eps^2}}^i-\eps B_{s_i}^i)\psi(x_i) 
\phi(\frac{t}{\eps^2}-s_i-r)\\
&\times \exp\Big\{\lambda^2\int_{-\infty}^r \int_{\R^d}\Phi_{t,\eps x_1-\eps B_{s_1}^1+y,B^1}^\eps(s',z)
\Phi_{t,\eps x_2-\eps B_{s_2}^2+y,B^2}^\eps(s',z)dzds'\Big\} 
{d\bar x dy d\bar s}.
\end{aligned}
\end{equation}
The exponent in the last line above can be written as
\begin{equation}\label{sep1518}
\begin{aligned}
&\int_{-\infty}^r \int_{\R^d}\Phi_{t,\eps x_1-\eps B_{s_1}^1+y,B^1}^\eps(s',z)\Phi_{t,\eps x_2-\eps B_{s_2}^2+y,B^2}^\eps(s',z)dzds'\\
&=\int_{-\infty}^r \int_{\R^d}\int_{[0,t/\eps^2]^2}\phi(\frac{t}{\eps^2}-u_1-s')\phi(\frac{t}{\eps^2}-u_2-s')
\psi(x_1-B_{s_1}^1+\farc{y}{\eps}+B_{u_1}^1-z)
 \\
&\times\psi(x_2-B_{s_2}^2+\farc{y}{\eps}+B_{u_2}^2-z)du_1du_2dzds'
\\
&=\int_{[0,{t}/{\eps^2}]^2}R_\phi(u_1+r-\frac{t}{\eps^2},u_2+r-\frac{t}{\eps^2})
R_\psi(x_1-x_2+\Delta B_{s_1,u_1}^1-\Delta B_{s_2,u_2}^2)
du_1du_2,
\end{aligned}
\end{equation}
with $R_\phi,R_\psi$ defined in \eqref{e.defco}. Next,  
we also integrate in the $r$-variable, with a change of variable $r\mapsto {r}/{\eps^2}$, so that 
\begin{equation}\label{sep1424}
\begin{aligned}
\E\left[\int_{{t_1}/{\eps^2}}^{{t_2}/{\eps^2}} \int_{\R^d} |Z^\eps_t(r,y)|^2 dydr\right]
=\eps^{d-2} \int_{t_1}^{t_2}\int_{\R^{3d}}\int_{[0,{t}/{\eps^2}]^2}\what{\E}_{B,{t}/{\eps^2}}[I e^{J}] \  d\bar s d\bar x
dy dr,
\end{aligned}
\end{equation}
with
\begin{equation}\label{sep1426}
\begin{aligned}
I&=\prod_{i=1}^2 g(\eps x_i+y-\eps B_{s_i}^i)u_0(\eps x_i+y+
\eps \Delta B_{s_i,\frac{t}{\eps^2} }^i)\phi(\frac{t-r}{\eps^2}-s_i)\psi(x_i),\\
J&=\lambda^2\int_{[0,{t}/{\eps^2}]^2}R_\phi(u_1-\frac{t-r}{\eps^2},u_2-\frac{t-r}{\eps^2})R_\psi(x_1-x_2+
\Delta B_{s_1,u_1}^1-\Delta B_{s_2,u_2}^2)du_1du_2.
\end{aligned}
\end{equation}
As $\phi$ is supported on $[0,1]$, the integration domain in $s_i$ is 
is $[\frac{t-r}{\eps^2}-1,\frac{t-r}{\eps^2}]$, 
because of the corresponding factor  
in the expression for $I$ in (\ref{sep1426}). A change of variable $s_i\mapsto {(t-r)}/{\eps^2}-s_i$ turns the domain of integration in $s_i$
into $[0,1]$, as in (\ref{e.varz}). It also turns $I$ in (\ref{sep1426}) into expression (\ref{e.defJ1}) for ${\cal I}_\eps$.
For the integral in~$u_i$ in the expression for $J$ in (\ref{sep1426}), 
 to have $R_\phi\neq 0$, we need 
$u_i\geq \frac{t-r}{\eps^2}-1$, so the integration domain for $u_i$ is $[ \frac{t-r}{\eps^2}-1,\frac{t}{\eps^2}]$. 
The change of variable  $u_i\mapsto \frac{t-r}{\eps^2}+u_i$  turns this 
into~$[-1,r/\eps^2]$, and $J$ into $\J_\eps(r/\eps^2,r/\eps^2)$.  
%\[
%I\mapsto \I_\eps,  \   \  J\mapsto \J_\eps(\frac{r}{\eps^2},\frac{r}{\eps^2}),
%\]
%which 
This completes the proof of (\ref{e.varz}). \end{proof}

\begin{remark}\label{r.intedomain}
The assumption $t_2\leq t-\eps^2$
in the statement of Lemma~\ref{l.varre} 
is only made to simplify the presentation of the result. 
For any $t_2\leq t$, a similar result holds -- we only need to 
modify the integration domain for $u_1,u_2$ in \eqref{e.defJ2} to 
$[-{(t-r)}/{\eps^2},{r}/{\eps^2}]^2$ and that of $s_1,s_2$ in \eqref{e.varz} to~$[0,{(t-r)}/{\eps^2}]^2$ -- 
this only makes a difference when $t-r\leq \eps^2$. 
\end{remark}

\section{Proof of Theorem~\ref{t.mainth}}\label{sec:thm-proof} 

%We first explain how the renormalization constants $c_1$ and $c_2$ are determined. 
%\begin{lemma}\label{l.renormalizationconstant}
%There exist $c_1,c_2$ such that 
%\begin{equation}\label{e.zetatc1c2}
%\zeta_t:=\log\E_B\Big[e^{\frac12\lambda^2\int_{[0,t]^2}R(s-u,B_s-B_u)dsdu}\Big]=c_1t+c_2+o(1),~~\hbox{as $t\to\infty$.}
%\end{equation}
%\end{lemma}
%{Lemma \ref{l.renormalizationconstant}, which explains the choice of $c_1$ and $c_2$,}
%is proved in Appendix~\ref{sec:append-tech}. {As a consequence of the lemma},
%By Lemma~\ref{l.renormalizationconstant}, it is sufficient to 

{We first prove the central limit theorem 
for the centered random fluctuation in \eqref{e.mainthCLT}, 
and then the leading order homogenization result in \eqref{e.conmean}. In the course of the proof, we replace the renormalization factor 
$e^{-c_1t/\eps^2-c_2}$ by $e^{-\zeta_{{t}/{\eps^2}}}$. The replacement
will be justified below 
in Lemma~\ref{l.renormalizationconstant} of the appendix.
}

\subsection*{Convergence of the fluctuations: the outline}

Fix a test function $g(x)\in C_c^\infty(\R^d)$, and go back to~(\ref{e.store})-\eqref{e.stochainte}. Our goal will be to show
that
the integrand $Z_t^\eps(r,y)$ depends mainly on $\dot{W}(s,\cdot)$ with $s$ close to $r$, so that
the stochastic integral is an approximate linear combination of 
strongly mixing processes,
which should satisfy
a central limit theorem. To make the ``local dependence'' more precise, we decompose 
the interval  $[-1,{t}/{\eps^2}]$ of integration in (\ref{e.store})
into alternating subintervals of size $\eps^{-\alpha}$ and $\eps^{-\beta}$ with $0<\alpha<\beta<2$:
\[
[-1,\frac{t}{\eps^2}]=[-1, \eps^{-\alpha}]\cup [\eps^{-\alpha},\eps^{-\beta}+\eps^{-\alpha}] 
\cup[\eps^{-\beta}+\eps^{-\alpha},\eps^{-\beta}+2\eps^{-\alpha}]\cup\ldots \cup [t_\eps,\frac{t}{\eps^2}],
\]
with $t_\eps$ chosen so that $|{t}/{\eps^2}-t_\eps|=O(\eps^{-\beta})$.

Denote the ``short'' intervals of  length
$\eps^{-\alpha}$ by $\{ I_{\alpha,j}\}$ and 
%\sout{those}
the ``long'' ones of  length
$\eps^{-\beta}$ by~$\{I_{\beta,j}\}$, and set
\[ 
I_\alpha=\bigcup_j I_{\alpha,j}, ~~I_\beta=\bigcup_jI_{\beta,j}.
\]
%\sout{we omit the dependence on $\eps$). }
The last piece $[t_\eps,{t}/{\eps^2}]$ is assigned to $I_{\alpha}$. We will define a modification~$\tilde{Z}_t^\eps(r,y)$
of~$Z_t^\eps(r,y)$ for $r\in I_\beta$, in \eqref{e.deftildez},
%\sout{denoted by $\tilde{Z}_t^\eps(r,y)$,} 
so that  $\tilde{Z}_t^\eps(r,y)$ 
only depends on~$\dot{W}(s,\cdot)$ with  $s\in (r-\eps^{-\alpha},r]$, and thus 
the random variables
\begin{equation}\label{e.defxj}
\cX_j^\eps:=\frac{1}{\eps^{{d}/{2}-1}}\int_{I_{\beta,j}}\int_{\R^d} \tilde{Z}_t^\eps(r,y)dW(r,y)
\end{equation}
are independent. To prove  the central limit theorem statement \eqref{e.mainthCLT} in Theorem~\ref{t.mainth}, 
it suffices to show that
%\\
%(1) The error induced by the modification on the ``long'' intervals is small: (Lemma~\ref{l.err1}):
{\begin{lemma}\label{l.err1}
We have
\begin{equation}\label{sep1512}
\frac{1}{\eps^{d-2}}\int_{I_\beta}\int_{\R^d} \E\left[|Z_t^\eps(r,y)-\tilde{Z}_t^\eps(r,y)|^2\right] dydr\to0~~~\hbox{ as $\eps\to0$.}
\end{equation}
\end{lemma}
%
%\begin{equation}\label{sep1508}
%\frac{1}{\eps^{d-2}}\int_{I_\beta}\int_{\R^d} \E[|Z_t^\eps(r,y)-\tilde{Z}_t^\eps(r,y)|^2] dydr\to 0\hbox{ as $\eps\to 0$.}
%\end{equation}
\begin{lemma}\label{l.err2} We have
\begin{equation}\label{sep1524}
\frac{1}{\eps^{d-2}}\int_{I_\alpha}\int_{\R^d} \E[|Z_t^\eps(r,y)|^2]dydr\to 0~~\hbox{ as $\eps\to0$.}
\end{equation}
\end{lemma}
%(2) The contribution from the ``short intervals'' $I_\alpha$ is small (Lemma~\ref{l.err2}):
%\begin{equation}\label{sep1510}
%\frac{1}{\eps^{d-2}}\int_{I_\alpha}\int_{\R^d} \E[|Z_t^\eps(r,y)|^2]dydr
%\to  0 \hbox{ as $\eps\to 0$.}
%\end{equation}
 \begin{lemma}\label{l.clt} We have
\[
\lambda\sum_j \cX_j^\eps\Rightarrow \int_{\R^d} \cU(t,x)g(x)dx~~\hbox{ in distribution as $\eps\to0$.}
\]
Here, $\cU(t,x)$ is the solution of (\ref{e.ewshe}). 
\end{lemma}}
%(3) The sum $\sum_j \cX_j^\eps$ satisfies a central limit theorem, 
%with variance given by (\ref{e.ewshe}) (Lemma~\ref{l.clt}).

\bigskip

\subsubsection*{The modification}

We first explain how the modification is done. Recall 
that 
\[
Z^\eps_t(r,y)=\int_{\R^d} g(x) \what{\E}_{B,{t}/{\eps^2}}
\Big[u_0(x+\eps B_{{t}/{\eps^2}})\Phi^\eps_{t,x,B}(r,y)
\exp\Big\{\lambda M^\eps_{t,x,B}(r)-\frac12\lambda^2\la M^\eps_{t,x,B}\ra_r\Big\} \Big] dx
\]
depends on  {$W$} only through the martingale in the exponent 
\begin{equation}\label{sep1430}
M_{t,x,B}^\eps(r)=\int_{-\infty}^r \int_{\R^d} \left(\int_0^{{t}/{\eps^2}}\phi(\frac{t}{\eps^2}-s'-s)\psi(\frac{x}{\eps}+B_{s'}-y)ds'\right)dW(s,y).
\end{equation}
Since $\phi$ is supported on $[0,1]$, the integration in $s'$ 
in (\ref{sep1430}) is 
only over $s'<t/\eps^2-s$ (in fact, over the interval
  $(t/\eps^2-s-1,t/\eps^2-s)$), so that 
\begin{equation}\label{sep1502}
  M_{t,x,B}^\eps(r)=\int_{-\infty}^r \int_{\R^d} \left(\int_{0}^{{t}/{\eps^2}-s}
\phi(\frac{t}{\eps^2}-s'-s)\psi(\frac{x}{\eps}+B_{s'}-y)ds'\right)dW(s,y).
\end{equation}
We expect that, because we deal with dimensions $d\geq 3$,
and therefore  the transience of Brownian motion yields
mixing, most of the contributions to 
$M_{t,x,B}^\eps(r)$ come from $s$ 
``macroscopically near'' $r$, so that~$0<r-s<\eps^{-\alpha}$, with some $\alpha\in(0,2)$.   
Thus, we set
\begin{equation}
r_\eps:=\frac{t}{\eps^2}-r+\frac{1}{2\eps^\alpha},%\frac{t}{\eps^2}-r-1+\eps^{-\alpha}.
\end{equation}
and define the modification of $M_{t,x,B}^\eps(r)$ on $I_\beta$ as
\begin{equation}\label{e.defm}
\tilde{M}^\eps_{t,x,B}(r):=\int_{-\infty}^r \int_{\R^d} \left(\int_{0}^{r_\eps}
\phi(\frac{t}{\eps^2}-s'-s)\psi(\frac{x}{\eps}+B_{s'}-y)ds'\right)dW(s,y),~~~r\in I_\beta.
\end{equation}
Note that for $r\in I_\beta$, we have $r\geq \eps^{-\alpha}$, hence $r_\eps<{t}/{\eps^2}$. 
Due to the dependence of $r_\eps$ on $r$, $\tilde{M}_{t,x,B}^\eps$ is not a martingale. Still, with some abuse of 
notation, we  write 
\[
\la \tilde{M}_{t,x,B}^\eps\ra_r:=\int_{-\infty}^r \int_{\R^d} \left(\int_{0}^{r_\eps}\phi(\frac{t}{\eps^2}-s'-s)\psi(\frac{x}{\eps}+B_{s'}-y)ds'\right)^2 dyds.
\]
Note that if $s\leq r-\eps^{-\alpha}$, then
\[
\frac{t}{\eps^2}-s'-s\geq \frac{t}{\eps^2}-r_\eps-r+\eps^{-\alpha}=\frac{1}{2\eps^\alpha}>1,
\]
so the integrand in (\ref{e.defm}) vanishes.  
Thus, $\tilde{M}^\eps_{t,x,B}(r)$ only depends on ${dW}(s,\cdot)$ for  
$s\in(r-\eps^{-\alpha},r]$. The corresponding
modification of ${Z}_t^\eps(r,y)$ is 
\begin{equation}\label{e.deftildez}
\tilde{Z}_t^\eps(r,y):=\int_{\R^d} g(x) \what{\E}_{B,{t}/{\eps^2}}\left[u_0(x+\eps B_{{t}/{\eps^2}})\Phi^\eps_{t,x,B}(r,y)
\exp\Big\{\lambda \tilde{M}^\eps_{t,x,B}(r)-\frac12\lambda^2\la \tilde{M}^\eps_{t,x,B}\ra_r\Big\} \right] dx,
\end{equation}
%\sout{then $\tilde{Z}_t^\eps(r,y)$ only} \blue{i
 {which} also depends only on ${dW}(s,\cdot)$ for 
$s\in (r-\eps^{-\alpha},r]$, 
and the integrals $\{\cX_j^\eps\}$ defined in \eqref{e.defxj} are independent random variables.

\subsubsection*{Proof of the central limit theorem  \eqref{e.mainthCLT}}
 
%that 
%\begin{equation}
%\begin{aligned}
%\I_\eps
%= \prod_{i=1}^2 g(\eps x_i+y-\eps B_{\frac{t-r}{\eps^2}-s_i}^i)u_0(\eps x_i+y+\eps B_{\frac{t}{\eps^2}}^i-\eps B_{\frac{t-r}{\eps^2}-s_i}^i),
%\end{aligned}
%\end{equation}
%and 
%\begin{equation}
%\begin{aligned}
%&\J_\eps(M_1,M_2)\\
%&=\lambda^2\int_{-1}^{M_1}\int_{-1}^{M_2}R_\phi(u_1,u_2)R_\psi(x_1-x_2+B_{\frac{t-r}{\eps^2}+u_1}^1-B_{\frac{t-r}{\eps^2}-s_1}^1-B_{\frac{t-r}{\eps^2}+u_2}^2+B_{\frac{t-r}{\eps^2}-s_2}^2)du_1du_2.
%\end{aligned}
%\end{equation}
Recall (\ref{e.varz}), written as
\begin{equation}\label{e.varzbis}
\begin{aligned}
&\frac{1}{\eps^{d-2}}\E\left[\int_{{t_1}/{\eps^2}}^{{t_2}/{\eps^2}} \int_{\R^d} |Z^\eps_t(r,y)|^2 dydr\right]= 
\int_{t_1}^{t_2}\int_{\R^d}\cF_\eps(r,y,\farc{r}{\eps^2},\frac{r}{\eps^2})dy dr.
%= \int_{t_1}^{t_2}\int_{\R^{3d}}\int_{[0,1]^2}\what{\E}_{B,{t}/{\eps^2}}
%\Big[\I_\eps e^{\J_\eps(\frac{r}{\eps^2},\frac{r}{\eps^2})}\Big] \prod_{i=1}^2 \phi(s_i)\psi(x_i)\  d\bar s d\bar xdy dr.
\end{aligned}
\end{equation}
%\sout{Recall 
%\corO{\eqref{e.defJ1} and \eqref{e.defJ2}.}}
Here, for $r\in[0,t]$, $y\in\R^d$ and $M_1,M_2\leq{r}/{\eps^2}$,  we have set 
\begin{equation}
\cF_\eps(r,y,M_1,M_2):=\int_{\R^{2d}}\int_{[0,1]^2}\what{\E}_{B,{t}/{\eps^2}}\Big[\I_\eps e^{\J_\eps(M_1,M_2)}\Big]
\prod_{i=1}^2 \phi(s_i)\psi(x_i)ds_1ds_1dx_1dx_2,
\end{equation}
with $\I_\eps$ and $\J_\eps$ defined in (\ref{e.defJ1}) and (\ref{e.defJ2}), respectively.

We state the following proposition and postpone its proof to Section~\ref{s.pp}.
The function $\bar g(t,x)$  in the proposition
is the solution of the effective diffusion equation
\begin{equation}\label{e.heathom1}
\partial_t \bar{g}=\frac12\nabla\cdot \bfa_{\mathrm{eff}}\nabla \bar{g}, \   \  \bar{g}(0,x)=g(x),
\end{equation}
where $\bfa_{\mathrm{eff}}$ is as in \eqref{e.defaeff} below.
\begin{proposition}\label{p.varcon}
For any $r\in(0,t), y\in\R^d$, as $\eps\to0$ and $M_1,M_2\to\infty$, 
\begin{equation}\label{e.conf}
 \cF_\eps(r,y,M_1,M_2) \to \nu_{\mathrm{eff}}^2  |\bar{g}(t-r,y)\bar{u}(r,y)|^2,
\end{equation}
where $\bar{u},\bar{g}$ solve \eqref{e.heathom} and \eqref{e.heathom1}, and $\nu_{\mathrm{eff}}$ is defined in \eqref{e.defnu}. In addition, for any $k>0$,
\begin{equation}\label{e.bdf}
|\cF_\eps(r,y,M_1,M_2)|\leq C(1\wedge |y|^{-k})
\end{equation}
for some constant $C>0$ independent of $\eps,r,M_1,M_2$.
\end{proposition}
%Proposition \ref{p.varcon} is instrumental in completing the 
{Next 
 we present the proofs of Lemmas 
 \ref{l.err1}, \ref{l.err2} and \ref{l.clt}, which
in turn imply \eqref{e.mainthCLT}.}
%First, we show that (\ref{sep1508}) holds: the total error induced by the modification on the ``long'' intervals is small.
%\begin{lemma}\label{l.err1}
%We have
%\begin{equation}\label{sep1512}
%\frac{1}{\eps^{d-2}}\int_{I_\beta}\int_{\R^d} \E\left[|Z_t^\eps(r,y)-\tilde{Z}_t^\eps(r,y)|^2\right] dydr\to0~~~\hbox{ as $\eps\to0$.}
%\end{equation}
%\end{lemma}

\begin{proof}[Proof of Lemma~\ref{l.err1}]
By Lemma~\ref{l.varre}, we have
\[
\begin{aligned}
\frac{1}{\eps^{d-2}}\int_{I_\beta} \int_{\R^d} \E[|Z_t^\eps(r,y)|^2] dydr=\int_0^t \int_{\R^d} 1_{\{{r}/{ \eps^2}\in I_\beta\}} \cF_\eps(r,y,\frac{r}{\eps^2},\frac{r}{\eps^2})dydr.
\end{aligned}
\]
The same calculation as in the proof of that lemma gives
\begin{equation}\label{sep1520}
\frac{1}{\eps^{d-2}}\int_{I_\beta} \int_{\R^d} \E[|\tilde{Z}_t^\eps(r,y)|^2] dydr=\int_0^t \int_{\R^d} 
1_{\{{r}/{ \eps^2}\in I_\beta\}} \cF_\eps(r,y,\frac{1}{2\eps^\alpha},\frac{1}{2\eps^\alpha})dydr,
\end{equation}
and
\begin{equation}\label{sep1522}
\frac{1}{\eps^{d-2}}\int_{I_\beta} \int_{\R^d} \E[Z_t^\eps(r,y)\tilde{Z}_t^\eps(r,y)] dydr=
\int_0^t \int_{\R^d} 1_{\{{r}/{ \eps^2}\in I_\beta\}} \cF_\eps(r,y,\frac{r}{\eps^2},\frac{1}{2\eps^\alpha})dydr.
\end{equation}
Indeed, the only required modification in replacing $M^\eps$ by $\tilde M^\eps$ is to replace the upper limit $t/\eps^2$
of integration in $s$ in (\ref{sep1514}) by $r_\eps$.  This leads to the same change of the upper limit of integration in $u$ in (\ref{sep1518}),
and in the expression for $J$ in (\ref{sep1426}). The changes of variables described below (\ref{sep1426}) then bring about (\ref{sep1520})
and (\ref{sep1522}). By Proposition~\ref{p.varcon}, the proof is complete, 
as (\ref{e.bdf}) allows us to apply 
the Lebesgue dominated convergence theorem.
\end{proof}

%The next step is to establish (\ref{sep1510}): the contribution of the ``short'' intervals is small.
%\begin{lemma}\label{l.err2} We have
%\begin{equation}\label{sep1524}
%\frac{1}{\eps^{d-2}}\int_{I_\alpha}\int_{\R^d} \E[|Z_t^\eps(r,y)|^2]dydr\to 0~~\hbox{ as $\eps\to0$.}
%\end{equation}
%\end{lemma}

\begin{proof}[Proof of Lemma~\ref{l.err2}]
By Lemma~\ref{l.varre}, we have 
\[
\begin{aligned}
\frac{1}{\eps^{d-2}}\int_{I_\alpha} \int_{\R^d} \E[|Z_t^\eps(r,y)|^2] dydr=\int_0^t \int_{\R^{d}}
1_{\{{r}/{ \eps^2}\in I_\alpha\}} \cF_\eps(r,y,\frac{r}{\eps^2},\frac{r}{\eps^2})dydr.
\end{aligned}
\]
Note that when $t-r\leq \eps^2$, the expressions for 
$\I_\eps,\J_\eps$, as well as $\cF_\eps$ are
slightly different, see Remark~\ref{r.intedomain}. In this case, it is easy to check 
that Proposition~\ref{p.varcon} still holds. The uniform bound~(\ref{e.bdf}), 
as well as the fact that  
\[
|\{r\in[0,t]:~r/\eps^2\in I_\alpha\}|\to 0\hbox{ as $\eps\to 0$},
\]
 complete the proof.
\end{proof}

%The last step in the proof of (\ref{e.mainthCLT}) 
%is to establish the central limit theorem  for the sums over the variables $\cX_j^\eps$ defined by
%(\ref{e.defxj}).
% \begin{lemma}\label{l.clt} We have
%\[
%\lambda\sum_j \cX_j^\eps\Rightarrow \int_{\R^d} \cU(t,x)g(x)dx~~\hbox{ in distribution as $\eps\to0$.}
%\]
%Here, $\cU(t,x)$ is the solution of (\ref{e.ewshe}). 
%\end{lemma}
\begin{proof}[Proof of Lemma~\ref{l.clt}]
First,  it is easy to check that the solution of (\ref{e.ewshe}) satisfies
\begin{equation}\label{e.exvar}
\Var\Big[\int_{\R^d}\mathscr{U}(t,x)g(x)dx\Big]=\lambda^2\nu_{\mathrm{eff}}^2 \int_0^t\int_{\R^d} | \bar{g}(t-s,x)\bar{u}(s,x)|^2 dxds.
\end{equation}
Let 
\[
s_{n,\eps}^2=\lambda^2\sum_j\Var[\cX_j^\eps],
\]
then, by the same calculation as in the proofs of Lemma~\ref{l.varre} and~\ref{l.err1}, we have 
\[
s_{n,\eps}^2=\lambda^2\sum_j\int_0^t \int_{\R^d} 1_{\{{r}/{ \eps^2}\in I_{\beta,j}\}} 
\cF_\eps(r,y,\frac{1}{2\eps^\alpha},\frac{1}{2\eps^\alpha})dydr\to  \Var\Big[\int_{\R^d} \cU(t,x)g(x)dx\Big].
\]
The last step comes from Proposition~\ref{p.varcon} and \eqref{e.exvar}. 

Since $\cX_j^\eps$ are independent random variables, it remains 
to check the Lindeberg condition 
which reduces in our case to: for any $\delta>0$,
\begin{equation}\label{e.lindeberg}
\sum_j \E[|\cX_j^\eps|^2 1_{\{|\cX_j^\eps|>\delta}]\to0
\end{equation}
as $\eps\to0$. By the
Cauchy-Schwarz and Chebyshev inequality, we have 
\[
\sum_j \E[|\cX_j^\eps|^2 1_{\{|\cX_j^\eps|>\delta}] \leq \sum_j \sqrt{\E[|\cX_j^\eps|^4]} \sqrt{\E[|\cX_j^\eps|^2]/\delta^2}
\leq \frac{1}{\delta}  \Big(\sum_j \sqrt{\E[|\cX_j^\eps|^4]}\Big)\Big( \sup_j \sqrt{\E[|\cX_j^\eps|^2]}\Big).
\]
Proposition~\ref{p.varcon} implies that for all $j$ we have
\[
\E[|\cX_j^\eps|^2]=\int_0^t \int_{\R^d} 1_{\{{r}/{ \eps^2}\in I_{\beta,j}\}} 
\cF_\eps(r,y,\frac{1}{2\eps^\alpha},\frac{1}{2\eps^\alpha})dydr \les \eps^{2-\beta}.
\]
Lemma~\ref{l.bdsum4} proved in Appendix~\ref{sec:append-tech} shows that
\[
\sum_j \sqrt{\E[|\cX_j^\eps|^4]}\les1,
\]
and (\ref{e.lindeberg}) follows.
%\sout{\corO{The last two displays complete the proof.}}
\end{proof}

\subsubsection*{Proof of the homogenization limit \eqref{e.conmean}}

The proof of (\ref{e.conmean}) is now straightforward.
We write 
\[
\int_{\R^d} u_\eps(t,x)e^{-\zeta_{t/\eps^2}}g(x)dx=\int_{\R^d} (u_\eps(t,x)-\E[u_\eps(t,x)])e^{-\zeta_{t/\eps^2}}g(x)dx
+\int_{\R^d} \E[u_\eps(t,x)]e^{-\zeta_{t/\eps^2}}g(x)dx.
\]
The first term goes to zero in probability by \eqref{e.mainthCLT}. For the second term, by Lemma~\ref{l.conmean} below, we have 
\[
\E[u_\eps(t,x)]e^{-\zeta_{t/\eps^2}}=\what{\E}_{B,{t}/{\eps^2}}[u_0(x+\eps B_{{t}/{\eps^2}})]\to \bar{u}(t,x),
\]
finishing the proof.~$\Box$ 

The rest of the paper is devoted to the proof of Proposition~\ref{p.varcon}, as well as the other auxiliary statements used in this section,
such as the technical lemmas in Appendix~\ref{sec:append-tech}.

\section{The tilted Brownian motion}
\label{s.chain}

The previous section relies on analyzing the expectations under the tilted measure
\[
\what{\E}_{B,{t}/{\eps^2}}[f(B)]=\E_B\Big[f(B) \exp\Big\{\frac12\lambda^2\int_{[0,{t}/{\eps^2}]}R(s-u,B_s-B_u)dsdu\Big\}\Big]
e^{-\zeta_{t/\eps^2}}.
\]
The goal of this section is to construct a Markov chain taking values in $\C([0,1])$ so that the tilted Brownian path 
on $\C([0,{t}/{\eps^2}])$ can be represented by the chain, and satisfies an invariance principle. 
We also analyze the intersection of two independent paths and show that the total ``intersection time'' has exponential tails.

\subsection{Construction of the Markov chain on $\C([0,1])$}\label{s.constructionchain}
For any $T>0$, let 
 \[
 \Omega_T=\{\omega: \omega \in \C([0,T]), \omega(0)=0\}
 \]
 be the configuration space. 
 Denoting
 the tilted measure by $\what{\Pb}_T$, and the Wiener measure by $\Pb_T$, 
 we have 
\begin{equation}\label{e.rndeT}
\frac{d\what{\Pb}_T}{d\Pb_T}(\omega)=\exp\Big\{\frac12\lambda^2\int_{[0,T]^2}R(s-u,\omega(s)-\omega(u))dsdu-\zeta_T\Big\}.
%\propto \ \exp\Big\{\lambda^2 \int_0^T\int_0^s R(s-u,\omega(s)-\omega(u))duds\Big\}.
\end{equation}
Define the probability space $(\Omega,\mathcal{A},\pi)$ 
with $\Omega=\Omega_1$, $\mathcal{\A}$ the Borel sigma-algebra on $\Omega_1$, 
$\pi=\what{\Pb}_1$, and denote the expectation by $\E_\pi$. We will decompose the 
path of length $T$ into  increments of 
length $1$ which take values in $\Omega$. In order to consider the distribution of 
the path on $[t,t+1]$ for any~$t>0$,   
we introduce a parameter $\tau\in(0,1]$,  
set $N=[T-\tau]$,
and
divide the interval $[0,T]$ into $N+2$ 
subintervals~$(\tau_k,\tau_{k+1})$,~$k=0,\dots,N+1$, with
$\tau_0=0$, $\tau_1=\tau$, $\tau_{k+1}=\tau_k+1$ for $k=1,\dots,N$, 
and~$\tau_{N+2}=T$.
The increments of the path on $(\tau_k,\tau_{k+1})$ are
denoted by $\{x_k\}$, with $x_0\in\Omega_\tau$,  ~$x_k\subset \Omega$ for  $k=1,\dots,N+1$,
and $x_{N+1}\in \Omega_{T-\tau-N}$.   Given $x_k$, we define 
$\omega_s$, $0\le s\le T$, as
\begin{equation}\label{e.constructpath}
\omega_s=\left\{\begin{array}{ll}
x_0(s)  & s\in [0,\tau], \\ 
\omega_{\tau+k-1}+x_k(s-\tau-k+1) & s\in [\tau+k-1,\tau+k],  k=1,\ldots,N,\\
\omega_{\tau+N}+x_{N+1}(s-\tau-N) & s\in [\tau+N,T],
\end{array}
\right.
\end{equation}
and write 
\[
\{\omega_s\}=(x_0,\ldots,x_{N+1}).
\] 
For $t>0$ and $t\notin \Z_{\geq 1}$, we only need to choose $\tau=t-[t]$ so that $\{\omega_s: s\in[t,t+1]\}=x_k$ for some $k$.
Write
\begin{equation}\label{sep1708}
\begin{aligned}
&\int_{[0,T]^2}R(s-u,\omega(s)-\omega(u))dsdu=\sum_{k,m=0}^{N+1}Q_{km},~~
Q_{km}=\int_{\tau_k}^{\tau_{k+1}}\int_{\tau_m}^{\tau_{m+1}}
R(s-u,\omega(s)-\omega(u))dsdu.
 \end{aligned}
\end{equation}
Since $R(s,\cdot)=0$ when $|s|\geq1$, 
we only have  
nearest-neighbor interactions of $(x_0,\ldots,x_{N+1})$ 
in~\eqref{sep1708}:~$Q_{km}=0$ unless $|m-k|\le 1$, and
\begin{equation}\label{sep1709}
\begin{aligned}
&\int_{[0,T]^2}R(s-u,\omega(s)-\omega(u))dsdu=\sum_{k=0}^{N+1}Q_{kk}+2\sum_{k=0}^N Q_{k,k+1}.
\end{aligned}
\end{equation}
For $k=1,\dots,N$ and $0\le s\le 1$, we can write
\[
\omega(\tau_k+s)=\omega(\tau_k)+x_k(s),~~\omega(\tau_k+1+s)=\omega(\tau_k)+x_k(1)+x_{k+1}(s),
\]
so that
\begin{equation}\label{sep1710}
\begin{aligned}
Q_{k,k+1}&=\int_{\tau_k}^{\tau_{k}+1}\int_{\tau_k+1}^{\tau_{k}+2}R(s-u,\omega(s)-\omega(u))dsdu\\
&=\int_0^1\int_0^1 R(s+1-u,\omega(\tau_k+1+s)-\omega(\tau_k+u)) dsdu\\
&=\int_0^1\int_0^1 R(s+1-u,x_k(1)+x_{k+1}(s)-x_k(u)) dsdu.
\end{aligned}
\end{equation}
Thus, for $x,y\in \Omega$, we define the interaction term
\begin{equation}\label{e.defI}
I(x,y)=\lambda^2\int_0^1\int_0^1 R(s+1-u,y(s)+x(1)-x(u))dsdu.
\end{equation}
The interactions between $x_0,x_1$ and that 
of $x_{N},x_{N+1}$ are defined slightly differently  as 
\[
\begin{aligned}
&I_{0,1}(x_0,x_1)=\lambda^2\int_0^\tau du \int_{0}^{1}ds \ R(s+\tau-u,x_1(s)+x_0(\tau)-x_0(u)),\\
&I_{N,N+1}(x_N,x_{N+1})=\lambda^2\int_0^1 du\int_0^{T-\tau-N}ds \ R(s+1-u,x_{N+1}(s)+x_N(1)-x_N(u)).
\end{aligned}
\]
It is now straightforward to check that
\begin{equation}\label{sep1712}
\what{\Pb}_T(d\omega)\propto \ \what{\Pb}_{\tau}(dx_0)e^{I_{0,1}(x_0,x_1)}\prod_{k=1}^{N-1}\pi(dx_k)e^{I(x_k,x_{k+1})}
%\\ &\times 
\pi(dx_N)e^{I_{N,N+1}(x_{N},x_{N+1})} \what{\Pb}_{T-\tau-N}(dx_{N+1}).
\end{equation}

The Krein-Rutman and Doob-Krein-Rutman theorems (see Appendix to Chapter VIII of~\cite{dautraylionsvol3}) imply
that there exist $\rho>0$ and $\Psi(y)$    solving the eigenvalue problem
\begin{equation}\label{e.coreq}
\int_{\Omega} e^{I(x,y)}\Psi(y)\pi(dy)=\rho \Psi(x),
\end{equation}
such that $\rho$ is the largest possible eigenvalue, 
\begin{equation}\label{sep1804}
0<c_1\le\Psi(y)\le c_2<+\infty~~\hbox{ for all $y\in\Omega$,}
\end{equation}
and $\Psi$ is the unique eigenvector associated with $\rho$,
normalized so that 
\begin{equation}\label{sep1808}
\int_{\Omega}\Psi(y)\pi(dy)=1.
\end{equation}
 Such an argument was also used in
   \cite{mukherjee2017central}. 
%\begin{remark}\label{r.bds}
%We mention that 
The bounds on $\rho$ and $\Psi$ only depend on $\|I\|_{L^\infty}$.
Indeed,~(\ref{sep1808}) implies that
\[
\rho=\int_{\Omega\times\Omega} e^{I(x,y)}\Psi(y)\pi(dx)\pi(dy),
\]
so we have 
\begin{equation}\label{sep1814}
e^{-\|I\|_\infty} \leq \rho\leq e^{\|I\|_\infty}.
\end{equation}
Since 
\[
\Psi(x)=\frac{1}{\rho}\int_{\Omega}e^{I(x,y)}\Psi(y)\pi(dy),
\]
we also have 
\begin{equation}\label{sep1816}
e^{-2\|I\|_\infty}\leq \Psi(x)\leq e^{2\|I\|_\infty}.
\end{equation}

Now we can re-write (\ref{sep1712}) as
\begin{equation}
\begin{aligned}
\what{\Pb}_T(d\omega)\propto &\  \what{\Pb}_{\tau}(dx_0)e^{I_{0,1}(x_0,x_1)}\Psi(x_1)\pi(dx_1)
\prod_{k=1}^{N-1}\hat{\pi}(x_k,dx_{k+1})
%\\ &\times 
\frac{e^{I_{N,N+1}(x_{N},x_{N+1})}}{\Psi(x_N)} \what{\Pb}_{T-\tau-N}(dx_{N+1}),
\end{aligned}
\end{equation}
with the transition probability density 
\begin{equation}\label{e.trankernel}
\hat{\pi}(x,dy)=\frac{e^{I(x,y)}\Psi(y)\pi(dy)}{\rho\Psi(x)}.
\end{equation}
Setting
\[
\begin{aligned}
&f_{0,1}(x_0)=\int_{\Omega} e^{I_{0,1}(x_0,x_1)}\Psi(x_1)\pi(dx_1),
~~~
f_{N,N+1}(x_N)=\int_{\Omega_{T-\tau-N}} e^{I_{N,N+1}(x_{N},x_{N+1})} \what{\Pb}_{T-\tau-N}(dx_{N+1}),  
\end{aligned}
\]
and
\[
\begin{aligned}
&\hat{\pi}_{0,1}(x_0,dx_1)=\frac{e^{I_{0,1}(x_0,x_1)}\Psi(x_1)\pi(dx_1)}{f_{0,1}(x_0)}, 
~~
\hat{\pi}_{N,N+1}(x_N,dx_{N+1})=\frac{e^{I_{N,N+1}(x_{N},x_{N+1})} \what{\Pb}_{T-\tau-N}(dx_{N+1})}{f_{N,N+1}(x_N)},
\end{aligned}
\]
we  obtain
\begin{equation}\label{e.rnde2}
\begin{aligned}
\what{\Pb}_T(d\omega)\propto & \ f_{0,1}(x_0)\what{\Pb}_{\tau}(dx_0) \left(\hat{\pi}_{0,1}(x_0,dx_1)\prod_{k=1}^{N-1}\hat{\pi}(x_k,dx_{k+1}) \hat{\pi}_{N,N+1}(x_N,dx_{N+1})\right)
%\\ &\times  
\frac{f_{N,N+1}(x_N)}{\Psi(x_N)}.
\end{aligned}
\end{equation}

Now, we construct the Markov chain $X_k$,  
with $X_0\in \Omega_{\tau}$, $\{X_k\}_{k=1}^N\subset\Omega$, 
and $X_{N+1}\in \Omega_{T-\tau-N}$, as follows:

(1) $X_0$ is sampled from the
(normalized) distribution $f_{0,1}(x_0)\what{\Pb}_{\tau}(dx_0)$,

(2) $(X_1,\ldots,X_{N+1})$ are sampled according to 
\[
\hat{\pi}_{0,1}(X_0,dx_1)\left(\prod_{k=1}^{N-1}\hat{\pi}(x_k,dx_{k+1})\right)\hat{\pi}_{N,N+1}(x_N,dx_{N+1}).
\]
 We construct the path $B$ by stitching together
all increments   as in \eqref{e.constructpath}:
\begin{equation}\label{e.BX}
B=\{B_s: s\in[0,T]\}=(X_0,\ldots,X_{N+1}).
\end{equation}
We use $\E_\pi$ to denote the expectation with respect to this Markov chain. In 
light of \eqref{e.rnde2}, for any~$F:\Omega_T\to\R$, 
we have the relation
\begin{equation}\label{sep1802}
\what{\E}_{B,T}[F(B)]:=\int_{\Omega_T} F(\omega)\what{\Pb}_T(d\omega)=
\E_\pi\left[F(B) c_{\tau,T}\frac{f_{N,N+1}(X_N)}{\Psi(X_N)}\right].
\end{equation}
Here, $c_{\tau,T}$ is the normalization constant:
\[
\frac{1}{c_{\tau,T}}:=\E_\pi\left[\frac{f_{N,N+1}(X_N)}{\Psi(X_N)}\right].
\]

Using (\ref{sep1814}) and (\ref{sep1816}),    
we see that the Doeblin condition is satisfied: there exists $\gamma\in(0,1)$ that depends only on $\|I\|_{L^\infty}$ such that 
\begin{equation}\label{e.doeblin}
\begin{aligned}
 \hat{\pi}(x,A)&\geq \gamma \pi(A)\\
\end{aligned}
\end{equation}
for all $x\in\Omega,A\in \mathcal{A}$. Therefore, there exists a unique invariant measure for $\hat{\pi}$, and
\begin{equation}\label{e.conexpfast}
d_{\mathrm{TV}}(X_k,\tilde{X})\les (1-\gamma)^k,
\end{equation} 
where $\tilde{X}$ is sampled from the invariant measure.

\subsubsection*{The two-component chain} 
To consider the interaction between two independent paths $B^1,B^2$,  we construct 
a two component Markov chain $Z_k=(X_k,Y_k)\in\Omega^2$ by sampling $X_k,Y_k$ independently. By the same discussion we have
\[
\begin{aligned}
B^1&=\{B^1_s: s\in[0,t/\eps^2]\}=(X_0,\ldots,X_{N_{\eps}+1}),\\
B^2&=\{B^2_s: s\in [0,t/\eps^2]\}=(Y_0,\ldots,Y_{N_{\eps}+1}),
\end{aligned}
\]
where $T=t/\eps^2$ and $N_\eps=[t/\eps^2-\tau]$. For any $F:\Omega_{t/\eps^2}\times \Omega_{t/\eps^2}\to \R$, we have 
\begin{equation}\label{e.rerk1}
\what{\E}_{B,{t}/{\eps^2}}[F(B^1,B^2)]=\E_{\pi}[F(B^1,B^2)\cG_{\eps}(X_{N_{\eps}})\cG_\eps(Y_{N_\eps})],
\end{equation}
with
\begin{equation}\label{e.rndeG}
\cG_{\eps}(X_{N_{\eps}}):=c_{\tau,{t}/{\eps^2}}\frac{f_{N_{\eps},N_{\eps}+1}(X_{N_{\eps}})}{\Psi(X_{N_{\eps}})}.
\end{equation}
Since both $f_{N_{\eps},N_{\eps}+1}$ 
and~$\Psi$ are bounded from above and below, 
and $\E_\pi[\cG_{\eps}(X_{N_\eps})]=1$, we know that $\cG_\eps$ is uniformly bounded in $\eps$.

For $k=2,\ldots,N_\eps$, $Z_k$ is sampled from $\hat{\bfpi}(Z_{k-1},dz_k)$ with 
\begin{equation}\label{e.2dtran}
\hat{\bfpi}(z_1,dz_2):=\hat{\pi}(x_1,dx_2)\hat{\pi}(y_1,dy_2), \   \ z_i=(x_i,y_i).
\end{equation}
%By (\ref{sep1804}), \eqref{e.trankernel},   \sout{$C^{-1}\leq \Psi(x)\leq C$} and the fact that 
%$I(x,y)$ \sout{being} \blue{is} uniformly bounded \blue{from above and below}, 
As for the single-component chain, the Doeblin condition is satisfied 
for $Z_k$ as well:
\begin{equation}\label{e.doeblin-twocomp}
\begin{aligned}
%\sup_{x\in\Omega, A\subset \Omega} \hat{\pi}(x,A)&\geq \gamma \  \pi(A)\\
\hat{\bfpi}(z,B)&\geq  \gamma  (\pi\times\pi)(B)
\end{aligned}
\end{equation}
for all $z\in\Omega^2,B\in \mathcal{A}\otimes\mathcal{A}$. After possibly decreasing the parameter $\gamma$, we can ensure that (\ref{e.doeblin}) 
and (\ref{e.doeblin-twocomp}) hold with the same~$\gamma\in(0,1)$. 

Writing 
\begin{equation}
\hat{\bfpi}(z_1,dz_2)=\gamma (\pi\times\pi)(dz_2)+(1-\gamma)\frac{\hat{\bfpi}(z_1,dz_2)-\gamma (\pi\times\pi)(dz_2)}{1-\gamma},
\end{equation}
we couple the two-component chain with a sequence of i.i.d. Bernoulli random variables $\eta_k$, $k\in\N$, with 
the parameter~$\gamma$: for $k=2,\ldots,N_\eps$, if $\eta_k=1$, we sample $Z_k$ from $(\pi\times \pi)(dz)$, and 
if $\eta_k=0$, we sample~$Z_k$ from 
\[
\frac{\hat{\bfpi}(Z_{k-1},dz)-\gamma (\pi\times\pi)(dz)}{1-\gamma},
\]
which is possible because of the Doeblin condition (\ref{e.doeblin-twocomp}). 
The same coupling works for the one-component chain, of course,
with the help of (\ref{e.doeblin}).
We enlarge the probability space so that $\eta_k $ are also defined on $(\Omega,\mathcal{A},\pi)$.

\subsection{The invariance principle for the tilted Brownian path}

We will use here the re-scaled version of (\ref{sep1802}): set
$T={t}/{\eps^2}$, $ N_{\eps}=[{t}/{\eps^2}-\tau]$,
and for any~$F:\Omega_{t/\eps^2}\to\R$ write 
\begin{equation}\label{e.rerk}
\what{\E}_{B,{t}/{\eps^2}}[F(B)]=\E_\pi\left[F(B)\cG_{\eps}(X_{N_{\eps}})\right],
\end{equation}
To simplify the notation, we kept the dependence on $\tau$ implicit in \eqref{e.rerk}.

We fix $\tau=1$ in this section, so that
 $N_\eps= [t/\eps^2]-1$,
\begin{equation}\label{e.BequalX}
B=\{B_s: s\in[0,t/\eps^2]\}=(X_0,\ldots,X_{N_\eps+1}).
\end{equation}
Here, $X_k$ is the increment of $B$ on $[k,k+1]$ for $k=0,\ldots,N_\eps$, and $X_{N_\eps+1}$ is the increment on 
the last interval $[[T],T]$. In this case, 
$X_0$ is sampled from $\Psi(x_0)\pi(x_0)$, 
$X_k$ is sampled from $\hat{\pi}(X_{k-1},dx_k)$ for $k=1,\ldots,N_\eps$, 
and $X_{N_\eps+1}$ is sampled from 
$\hat{\pi}_{N_\eps,N_\eps+1}(X_{N_\eps},dx_{N_\eps+1})$. 

For $k=1,\ldots,N_\eps$, 
we take 
independent Bernoulli random variables $\eta_k$ with 
parameter~$\gamma\in(0,1)$ as in the Doeblin condition, 
and  {consider} the regeneration {times} 
\begin{equation}
T_0=0, \   \ 
T_i=\inf\{j>T_{i-1}: \eta_j=1\}, \  \ i\geq 1.
\end{equation}
We define the path increment in each regeneration block as
 \[
\X_j:=\sum_{k=T_j}^{T_{j+1}-1}X_k(1), \   \  j=0,1,\ldots
\]
\begin{proposition}\label{p.clt}
For any $t>0$, 
\[
\eps B_{{s}/{\eps^2}}\Rightarrow  W_s
\]
in $\C([0,t])$ in $(\Omega,\A,\pi)$, where $W_s$ is a Brownian motion with the covariance matrix $\bfa_{\mathrm{eff}}$:
\begin{equation}\label{e.defaeff}
\bfa_{\mathrm{eff}}:=\gamma\E_\pi[\X_1\X_1^t].
\end{equation}
\end{proposition}
It is a straightforward computation to check that $\bfa_{\mathrm{eff}}$ in (\ref{e.defaeff}) does not depend on $\gamma$. In fact,
the right side of (\ref{e.defaeff}) can be written as 
\begin{equation}\label{sep2202}
\bfa_{\mathrm{eff}}=\lim_{n\to+\infty}\farc{1}{n} \E_\pi\left[(X_1(1)+\dots+X_n(1))(X_1(1)+\dots+X_n(1))^t\right].
\end{equation}

\begin{proof}
We show in Lemma \ref{l.meanvar} that $\X_1$ has 
zero mean and exponential tails, and further that  the random variables
\[
Z_i=\max_{s\in [T_i,T_{i+1}]}|B_s-B_{T_i}|,~~i\ge 0, 
\]
have exponential tails, and for $i\ge 1$ they are i.i.d.
From the first fact, one obtains
by Donsker's invariance
principle that 
\[
Y_n(t):=n^{-1/2} \sum_{j=1}^{[nt]}\X_i
\]
converges weakly to a Brownian motion with  
diffusivity $\E_\pi[\X_1\X_1^t]$. 
On the other hand, 
$T_n/n$ converges a.s. to $1/\gamma$, on 
account of the independence of the increments $T_i-T_{i-1}$  and the fact that 
they have  
mean $1/\gamma$ and are geometrically distributed. Setting 
\[
N_t^\eps=\max\{i: T_i<t/\eps^2\}-1,
\]
we deduce from \cite[Theorem 14.4]{billingsley} that the process
\[
\eps \sum_{i=1}^{N_t^\eps}\X_i
\]
converges in distribution
to a Brownian motion with the
diffusivity $\bfa_{\mathrm{eff}}$ given by (\ref{e.defaeff}).
On the other hand, we have
$$\max_{s\leq t} |\eps B_{s/\eps^2}-\eps \sum_{i=1}^{N_s^\eps}\X_i|
\les \eps \max_{i=1}^{N_t^\eps+1} |Z_i| \to_{\eps\to 0} 0\,, \quad a.s.,$$
because of the exponential tails of the $Z_i$. This completes the proof. 
\end{proof}

With the invariance principle, we can show the convergence of the average of
the solution. 

\begin{lemma}\label{l.conmean} We have
$\E[u_\eps(t,x)]e^{-\zeta_{t/\eps^2}}\to \bar{u}(t,x)$ as $\eps\to0$.
\end{lemma}
\begin{proof}
%\blue{
We first show that
\begin{equation}\label{sep1820}
\E_\pi[|\eps B_{{t_2}/{\eps^2}}-\eps B_{{t_1}/{\eps^2}}|^2] 
\leq C (t_2-t_1)
\end{equation}
with a constant $C>0$ independent of $0\leq t_1<t_2\leq t$ and $\eps>0$. 
Define 
\[
K_{1,\eps}=\min\{i: \frac{t_1}{\eps^2}<T_i<\frac{t_2}{\eps^2}\}, \   \   K_{2,\eps}=\max\{i: \frac{t_1}{\eps^2}<T_i<\frac{t_2}{\eps^2}\},
\]
and if there is no regeneration time in 
$({t_1}/{\eps^2},{t_2}/{\eps^2})$, we define 
$T_{K_{1,\eps}}={t_1}/{\eps^2}$ and $T_{K_{2,\eps}}={t_2}/{\eps^2}$. We decompose
\[
\eps B_{{t_2}/{\eps}^2}-\eps B_{{t_1}/{\eps^2}}=
(\eps B_{T_{K_{1,\eps}}}-\eps B_{{t_1}/{\eps^2}})+
(\eps B_{T_{K_{2,\eps}}}-\eps B_{T_{K_{1,\eps}}})+
(\eps B_{{t_2}/{\eps^2}}-\eps B_{T_{K_{2,\eps}}})
:=I_1+I_2+I_3.
\]
For $I_2$, we write 
\[
I_2=\eps \sum_{j=K_{1,\eps}}^{K_{2,\eps}-1} \X_j,
\]
and, conditioning on all the regeneration times, denoted by $\{T_i\}$,  
we obtain
\[
\begin{aligned}
\E_\pi[|I_2|^2 \ | \ \{T_i\}] = \eps^2 
\sum_{j=K_{1,\eps}}^{K_{2,\eps}-1} \E_\pi[\X_j^2| \{T_i\}].
\end{aligned}
\]
Here, we used the fact that $\X_j$ are independent with zero 
mean conditioning on $\{T_i\}$. By Lemma~\ref{l.meanvar}, we have
\[
\E_\pi[\X_j^2|\{T_i\}]\les (T_{j+1}-T_j)^2.
\]
As $K_{2,\eps}-K_{1,\eps}\leq \frac{t_2-t_1}{\eps^2}$, it follows that
\[
\begin{aligned}
\E_\pi[|I_2|^2 ] \les\eps^2 \E_\pi
\sum_{j=K_{1,\eps}}^{K_{2,\eps}-1}(T_{j+1}-T_j)^2\le 
C\eps^2\frac{t_2-t_1}{\eps^2}=C(t_2-t_1).
\end{aligned}
\]
Estimating the terms $I_1$ and $I_3$ is also straightforward using
Lemma~\ref{l.meanvar}, finishing the proof of~(\ref{sep1820}).

Next, note that by \eqref{e.rerk}, we have 
\[
\begin{aligned}
\E[u_\eps(t,x)]e^{-\zeta_{t/\eps^2}}=&\what{\E}_{B,{t}/{\eps^2}}[u_0(x+\eps B_{{t}/{\eps^2}})]
=\E_\pi[u_0(x+\eps B_{{t}/{\eps^2}})\cG_\eps(X_{N_\eps})]\\
=&\E_\pi[u_0((x+\eps B_{{t}/{\eps^2}-{1}/{\eps}})+\eps ( B_{{t}/{\eps^2}}-B_{{t}/{\eps^2}-{1}/{\eps}}))
\cG_\eps(X_{N_\eps})].
\end{aligned}
\]
Using (\ref{sep1820}), 
it suffices to consider 
\[
\E_\pi[u_0(x+\eps B_{{t}/{\eps^2}-1/\eps}) \cG_\eps(X_{N_\eps})].
\]
We apply Lemma~\ref{l.rnde} and Proposition~\ref{p.clt} to see that
\[
\E_\pi[u_0(x+\eps B_{{t}/{\eps^2}-{1}/{\eps}})\cG_\eps(X_{N_\eps})]-
\E_\pi[u_0(x+\eps B_{{t}/{\eps^2}-{1}/{\eps}})]\to0,
\]
and 
\[
\E_\pi[u_0(x+\eps B_{{t}/{\eps^2}-{1}/{\eps}})]\to\bar{u}(t,x),
\]
which completes the proof.
\end{proof}

\subsection{Intersection of independent paths}

The previous section shows that the tilted Brownian path behaves like a Brownian motion with an effective diffusivity, and this has been used to prove the convergence of 
\[
\begin{aligned}
\int_{\R^d} \E[u_\eps(t,x)]e^{-\zeta_{t/\eps^2}}g(x)dx.%&=\int_{\R^d} \what{\E}_{B,\frac{t}{\eps^2}}[u_0(x+\eps B_{\frac{t}{\eps^2}}] g(x)dx\\
%&\to \int_{\R^d}\bar{u}(t,x)g(x)dx.
\end{aligned}
\]
To control the variance of 
\[
\int_{\R^d} u_\eps(t,x)e^{-\zeta_{t/\eps^2}} g(x)dx,
\]
it is necessary to consider two independent tilted Brownian paths. We will show 
that the two paths can not intersect too 
much 
 -- this is the goal of this 
section and is only true  in dimensions $d\geq 3$. In fact, the proof of Proposition~\ref{p.expintegrability} below is the only place in the paper where we explicitly use the condition $\lambda\ll1$ and $d\geq 3$. Proposition~\ref{p.clt} and Lemma~\ref{l.conmean} hold in all dimensions and for all coupling constants.

For the sake of simplicity of 
presentation, we consider here only a homogeneous chain, 
assuming that $t/\eps^2$ is an integer, to avoid dealing
with the last step of the chain that has a different law. A~modification for a general $t$ is straightforward.  
Given any~$Z_0=(X_0,Y_0)\in \Omega^2$, we generate the chain 
$Z_k=(X_k,Y_k)$
according to the transition kernel~$\hat{\bfpi}$ defined in~\eqref{e.2dtran}.
The two components $X_k$ and $Y_k$ generate
two paths, that we denote by $\omega_{X_0},\omega_{Y_0}\in \C([0,\infty))$, 
via~\eqref{e.constructpath}. We recall that the regeneration times 
are defined as 
\begin{equation}
T_0=0, \   \ 
T_i=\inf\{j>T_{i-1}: \eta_j=1\}, \  \ i\geq 1,
\end{equation}
where $\eta_j$ are i.i.d Bernoulli random variables with 
parameter $\gamma\in(0,1)$.
    
Throughout the section, $X_0,Y_0$ are fixed, so we simply write $\pi[\cdot|X_0,Y_0]=\pi[\cdot]$. Define 
\[
\ell(x,y,X_0,Y_0)=\int_0^\infty1_{\{|x+\omega_{X_0}(s)-y-\omega_{Y_0}(s)|\leq 1\}}ds
\]
as the total ``nearby  time'' 
of $\omega_{X_0}$ and $\omega_{Y_0}$. We have the following result.

\begin{proposition}\label{p.expintegrability}
In $d\geq 3$, there exist constants $C_1,C_2>0$ such that 
\begin{equation}\label{e.insetimetail}
\sup_{x,y\in\R^d}\sup_{X_0,Y_0\in\Omega}
\pi[\ell(x,y,X_0,Y_0)>t] \leq C_1e^{-C_2 t}.
\end{equation}
As a consequence, if $\lambda<C_2$, then
\[
\sup_{x,y\in\R^d} \sup_{X_0,Y_0\in\Omega}\E_\pi[e^{\lambda \ell(x,y,X_0,Y_0)}]<\infty.
\]
\end{proposition}

\begin{proof}
The proof is divided into two steps. 

\emph{Step 1}. We show that there exists $K>0$ such that 
\begin{equation}\label{e.pr12}
\pi[\ell(x,y,X_0,Y_0)>K]<\frac12
\end{equation}
for all $x,y,X_0,Y_0$. Since
\[
\pi[\ell(x,y,X_0,Y_0)>K]\leq \pi[\ell(x,y,X_0,Y_0)>T_N]+\pi[T_N>K],
\]
with $T_N$ the $N-$th regeneration time, we only need to show
\begin{equation}\label{e.pr14}
 \pi[\ell(x,y,X_0,Y_0)>T_N]<\frac14
\end{equation}
for some $N$ independent of $x,y,X_0,Y_0$, and choose $K$ so
large that $\pi[T_N>K]<1/4$. To this end, it suffices to show 
that
\begin{equation}\label{sep1826}
\pi[E_N]<\frac14,~~\hbox{where}~~E_N=\big\{
\min_{s\geq T_N}|x+\omega_{X_0}(s)-y-\omega_{Y_0}(s)|\leq 1\big\}.
\end{equation}
Recall that 
\[
\omega_{X_0}(T_k)-\omega_{Y_0}(T_k)= 
\sum_{j=0}^{k-1} ( \X_j-\Y_j),~~~
\X_j-\Y_j=\sum_{i=T_j}^{T_{j+1}-1}[X_i(1)-Y_i(1)],~~\   \  k\geq 1.
\]
By the regeneration structure, $\X_j-\Y_j$ are i.i.d. random variables 
and are  also independent of~$\X_0-\Y_0$. 
For any $\alpha>0$, define 
\[
A_k=\{|x+\omega_{X_0}(T_k)-y-\omega_{Y_0}(T_k)| \leq k^\alpha\},
~~~A(N)=\bigcup_{k\geq N}A_k,
\]
%and $A(N)=\cup_{k\geq N}A_k$, 
and write
\[
\begin{aligned}
\pi[E_N]\leq &\pi[A(N)]+\pi[E_N\cap A(N)^c ] 
\leq &\sum_{k=N}^\infty \pi[A_k]+\pi[E_N\cap A(N)^c].
\end{aligned}
\]
By the local limit theorem in \cite[Theorem on p. 1]{stone1965local}, we have 
\begin{equation}\label{sep1902}
\pi[A_k] \leq C \frac{k^{\alpha d}}{k^{{d}/{2}}}=\frac{C}{k^{(\frac12-\alpha)d}}
\end{equation}
for some constant $C$ independent of $x,y,X_0,Y_0$. Thus, we can 
choose $\alpha<1/2-{1}/{d}$ (in $d\geq 3$) and~$N$ so large that  
\[
\sum_{k=N}^\infty \pi[A_k]<\frac18.
\]
On the other hand, we have
\[
\pi[E_N\cap A(N)^c ] \leq \sum_{k\geq N} \pi[B_k ],
\]
with 
\[
B_k:=A_k^c\cap \{ \min_{s\in[T_k,T_{k+1}]}|x+\omega_{X_0}(s)-y-\omega_{Y_0}(s)| \leq 1\},
\]
and $B_k\subset B_{k,X}\cup B_{k,Y}$ with 
\[
B_{k,X}=\left\{ \sum_{i=T_k}^{T_{k+1}-1} \max_{s\in [0,1]} |X_i(s)| >\frac{k^\alpha}{3}\right\},  \   \   B_{k,Y}=\left\{ \sum_{i=T_k}^{T_{k+1}-1} \max_{s\in [0,1]} |Y_i(s)| >\frac{k^\alpha}{3}\right\}.
\]
By Lemma~\ref{l.meanvar}, the random variable 
\[
\sum_{i=T_k}^{T_{k+1}-1} \max_{s\in [0,1]} |X_i(s)|
\]
has an exponential tail, which implies
that
\[
\pi[E_N\cap A(N)^c] \leq \sum_{k\geq N} e^{-Ck^\alpha}<\frac18
\]
 when $N$ is large. The proof of \eqref{e.pr14} is complete.

\emph{Step 2}. We define a sequence of stopping times as follows: $\tau_0=0$ and 
\[
\begin{aligned}
%&\tau_1=\min\{n\geq 1: \int_0^{n+1} 1_{\{|x+\omega_{X_0}(s)-y-\omega_{Y_0}(s)|\leq 1\}}ds>K\},\\
\tau_k=\min\Big\{n> \tau_{k-1}: 
\int_{\tau_{k-1}}^{n+1} 1_{\{|x+\omega_{X_0}(s)-y-\omega_{Y_0}(s)|\leq 1\}}ds>K
\Big\}, \  \ k\geq 1,
\end{aligned}
\]
with  $K$ chosen as in step 1. 
Let $n=[{t}/{K}]$, and  apply \eqref{e.pr12} to obtain
\[
\pi[\ell(x,y,X_0,Y_0)>t] \leq \pi[\tau_n <\infty] 
=\pi[\tau_n<\infty|\tau_1<\infty]\pi[\tau_1<\infty]\leq 
\frac12\pi[\tau_n<\infty |\tau_1<\infty].
\]
We  consider
\[
\begin{aligned}
\pi[\tau_2<\infty |X_{\tau_1},Y_{\tau_1}]=
\pi\Big[\int_{\tau_1}^\infty 1_{\{|x+\omega_{X_0}(s)-y-\omega_{Y_0}(s)|\leq 1\}}ds
>K |X_{\tau_1},Y_{\tau_1}\Big],
\end{aligned}
\]
and write for $s\ge\tau_1$:
\[
x+\omega_{X_0}(s)-y-\omega_{Y_0}(s)=
x+\omega_{X_0}(\tau_1)+[\omega_{X_0}(s)-\omega_{X_0}(\tau_1)]
-y-\omega_{Y_0}(\tau_1)-[\omega_{Y_0}(s)-\omega_{Y_0}(\tau_1)]. 
\]
Conditioning on $X_{\tau_1},Y_{\tau_1}$ gives
\[
\begin{aligned}
(\omega_{X_0}(\tau_1+\cdot )-\omega_{X_0}(\tau_1),\omega_{Y_0}(\tau_1+\cdot)-\omega_{Y_0}(\tau_1)) 
\stackrel{\text{law}}{=}  (\tilde{\omega}_{X_{\tau_1}}(\cdot ),\tilde{\omega}_{Y_{\tau_1}}(\cdot )),
\end{aligned}
\]
where $\tilde{\omega}$ is independent of $\omega$.
Hence, we may apply \eqref{e.pr12} again to get 
\[
\begin{aligned}
&\pi\Big[
\int_{\tau_1}^\infty 1_{\{|x+\omega_{X_0}(s)-y-\omega_{Y_0}(s)|\leq 1\}}ds >K |
X_{\tau_1},Y_{\tau_1}\Big]\\
&=\pi\Big[\int_0^\infty 1_{\{x+\omega_{X_0}(\tau_1)
+\tilde{\omega}_{X_{\tau_1}}(s)-y-\omega_{Y_0}(\tau_1)
-\tilde{\omega}_{Y_{\tau_1}}(s)\}}ds>K|X_{\tau_1},Y_{\tau_1}\Big]<\frac12
\end{aligned}
\]
uniformly in $x,y,X_{\tau_1},Y_{\tau_1}$. 
Iterating the same argument gives 
\[
 \pi[\ell(x,y,X_0,Y_0)>t ] \leq \left(\frac12\right)^2 \pi[\tau_n<\infty |\tau_2<\infty]
 \leq \ldots \leq \left(\frac12\right)^n,
\]
 which completes the proof.
\end{proof}

\begin{corollary}\label{c.intersection}
In $d\geq 3$, there exists $\lambda_0$ only depending on $\phi,\psi$ such that for $\lambda<\lambda_0$, we have
\[
\sup_{x,y\in\R^d}\sup_{(X_0,Y_0)\in\Omega^2}\E_\pi\Big[
\exp\Big\{\lambda \int_0^\infty\int_0^\infty  
R_\phi(u_1,u_2)R_\psi(x-y+\omega_{X_0}(u_1)-\omega_{Y_0}(u_2))du_1du_2\Big\} \Big]<\infty.
\]
\end{corollary}

\begin{proof}
As $R_\phi(u_1,u_2)=0$ if $|u_1-u_2|>1$ and $R_\psi$ is supported on $\{x:|x|\leq 1\}$, we have 
\[
\begin{aligned}
 &\int_0^\infty\int_0^\infty  R_\phi(u_1,u_2)R_\psi(x-y+\omega_{X_0}(u_1)-\omega_{Y_0}(u_2))du_1du_2\\
&\les \int_0^\infty\int_0^\infty  1_{\{|u_1-u_2|\leq 1\}}1_{\{ |x-y+\omega_{X_0}(u_1)-\omega_{Y_0}(u_2)|\leq 1\}} du_1du_2.
 \end{aligned}
 \]
Consider the region $u_2>u_1$. After a change of variable and an application of Jensen's inequality, we have
 \[
\begin{aligned}
&\exp\Big\{\int_0^1  \Big(\int_0^\infty \ 1_{\{ |x-y+\omega_{X_0}(u_1)-\omega_{Y_0}(u_1+u_2)|\leq 1\}}du_1 \Big)du_2 \Big\}\\
&\leq\int_0^1 \left( \exp\Big\{\int_0^\infty  1_{\{ |x-y+\omega_{X_0}(u_1)-\omega_{Y_0}(u_1+u_2)|\leq 1\}}du_1 \Big\}\right)du_2.
\end{aligned}
\]
It suffices to show that there exists $\lambda_0>0$ so that for 
$\lambda\in(0,\lambda_0)$ 
we have 
\begin{equation}\label{sep1904}
\E_\pi[e^{\lambda \ell(u_2,x,y,X_0,Y_0)}]\hbox{ is bounded uniformly in $u_2\in[0,1],x,y\in\R^d, X_0,Y_0\in\Omega$,}
\end{equation}
where 
\[
\ell(u_2,x,y,X_0,Y_0)= \int_0^\infty   1_{\{ |x-y+\omega_{X_0}(u_1)-\omega_{Y_0}(u_1+u_2)|\leq 1\}}du_1
\]
is the total ``nearby''  time of $\omega_{X_0}$ and the ``shifted'' $\omega_{Y_0}$. 
We can repeat the proof of~\eqref{e.insetimetail} verbatim to 
establish an identical estimate for  $\ell(u_2,x,y,X_0,Y_0)$,
from which (\ref{sep1904}) follows immediately, for $0<\lambda<C_2$. This
completes the proof.
\end{proof}

\section{Proof of Proposition~\ref{p.varcon}}
\label{s.pp}

Before proving Proposition~\ref{p.varcon}, we  
discuss  some heuristics of the convergence 
of $\cF_\eps(r,y,M_1,M_2)$ as~$\eps\to0$ and $M_1,M_2\to\infty$. Recall that
\begin{equation}\label{sep1906}
\cF_\eps(r,y,M_1,M_2)=\int_{\R^{2d}}\int_{[0,1]^2}\what{\E}_{B,{t}/{\eps^2}}\Big[\I_\eps e^{\J_\eps(M_1,M_2)}\Big]
\prod_{i=1}^2 \phi(s_i)\psi(x_i)ds_1ds_2dx_1dx_2,
\end{equation}
with
\begin{equation}
\I_\eps=\I_\eps(x_1,x_2,y,s_1,s_2,r)= 
\prod_{i=1}^2 g(\eps x_i+y-\eps B_{{(t-r)}/{\eps^2}-s_i}^i)u_0(\eps x_i+y+\eps B_{{t}/{\eps^2}}^i-\eps B_{{(t-r)}/{\eps^2}-s_i}^i).
\end{equation}
As shown in Proposition~\ref{p.clt}, the diffusively rescaled Brownian path $\eps B_{s/\eps^2}^i$ behaves like $W_s^i$, so we expect
that 
\begin{equation}\label{sep1908}
\I_\eps \Rightarrow \prod_{i=1}^2 g(y-W_{t-r}^i) u_0(y+W_t^i-W_{t-r}^i).
\end{equation}
in distribution. The exponential factor in (\ref{sep1906}) is
\begin{equation}\label{e.expJ}
\begin{aligned}
&\mathcal{J}_\eps(M_1,M_2)=\J_\eps(M_1,M_2,x_1,x_2,s_1,s_2,r)\\
&=\lambda^2\int_{-1}^{M_1}\int_{-1}^{M_2}R_\phi(u_1,u_2)R_\psi(x_1-x_2+B_{\frac{t-r}{\eps^2}+u_1}^1-B_{\frac{t-r}{\eps^2}-s_1}^1-B_{\frac{t-r}{\eps^2}+u_2}^2+B_{\frac{t-r}{\eps^2}-s_2}^2)du_1du_2,
\end{aligned}
\end{equation}
and 
measures the ``nearby''  time of two independent paths. Since $R_\psi$ is 
compactly supported,
most of the contribution in (\ref{e.expJ}) comes  
from $u_1,u_2\in [-1,M]$, with some large $M$ fixed, 
as indicated by Corollary~\ref{c.intersection}.
Thus, $\J_\eps$ depends only on the microscopic increments of $B^{1,2}$ around $(t-r)/\eps^2$ that are asymptotically
decorrelated from both $W_{t-r}^{1,2}$ and $W_t^{1,2}$.
Thus, $\J_\eps$ should be asymptotically 
independent  from  
$\I_\eps$, and the limit of $\J_\eps$ determines the effective variance $\nu_{\mathrm{eff}}^2$ in~\eqref{e.conf}.

The goal of this section is to make the above heuristics precise. 
The proof is in two steps.
We first show the convergence of $\cF_\eps$ for a fixed $r\in(0,t),y\in\R^d$. 
Then, 
we prove a uniform bound on $\cF_\eps$.

The expression \eqref{e.expJ} shows that $\J_\eps$ depends on the trajectories of $B^1,B^2$ starting 
from ${(t-r)}/{\eps^2}-1$, and
for a fixed $r\in(0,t),\eps>0$, we choose 
\[
\tau=\frac{t-r}{\eps^2}-\Big[\frac{t-r}{\eps^2}\Big].
\] 
Recall that $T=t/\eps^2$, $N_\eps=[t/\eps^2-\tau]$, and 
\[
\begin{aligned}
B^1&=\{B_s^1: s\in[0,t/\eps^2]\}=(X_0,\ldots,X_{N_{\eps}+1}),\\
B^2&=\{B_s^2: s\in[0,t/\eps^2]\}=(Y_0,\ldots,Y_{N_{\eps}+1}).
\end{aligned}
\]
It is clear that $\J_\eps$ 
is determined by the increments of~$B^1$ and $B^2$
for times larger than $(t-r)/\eps^2-2$, that is, for $n>N_{\eps,r}$, with
%\sout{only depends on the increment\blue{s} of~$B^1,B^2$ 
%on $[{(t-r)}/{\eps^2}-2,{(t-r)}/{\eps^2}-1]$, i.e., $X_{N_{\eps,r}},Y_{N_{\eps,r}}$ with }
\[
N_{\eps,r}=\Big[\frac{t-r}{\eps^2}\Big]-1.
\]
To simplify the notation, we define 
\[
\tilde{X}_\eps=X_{N_{\eps,r}}, \   \ \tilde{Y}_\eps=Y_{N_{\eps,r}}.
\]
We also note that by \eqref{e.rerk1}, we have 
\[
\what{\E}_{B,{t}/{\eps^2}}[\I_\eps e^{\J_\eps}]=\E_{\pi}[\I_\eps e^{\J_\eps}\cG_\eps(X_{N_\eps})\cG_\eps(Y_{N_\eps})].
\]

\subsection{Pointwise convergence} 

%\sout{Recall that} 
%\[
%\xcancel{\cF_\eps(r,y,M_1,M_2)=\int_{\R^{2d}}\int_{[0,1]^2}\what{\E}_{B,\frac{t}{\eps^2}}[\I_\eps e^{\J_\eps(M_1,M_2)}]\prod_{i=1}^2 \phi(s_i)\psi(x_i)ds_1ds_2dx_1dx_2.}
%\]

We first explain how the effective variance $ \nu_{\mathrm{eff}}$ is defined. 
For any ``starting pieces'' $X_0,Y_0\in\Omega$, and starting points
$x_1,x_2\in\R^d$, as well as
$M_1,M_2>0$,   and $s_1,s_2\in[0,1]$, we define
\begin{equation}\label{e.defFM}
\begin{aligned}
&\mathscr{H}_{M_1,M_2}(X_0,Y_0,x_1,x_2,s_1,s_2)=\E_\pi\Big[\exp\Big\{\lambda^2\int_{-1}^{M_1}\int_{-1}^{M_2}R_\phi(u_1,u_2)\\
&\times
R_\psi(x_1-x_2+\omega_{X_0}(2+u_1)-\omega_{X_0}(2-s_1)-\omega_{Y_0}(2+u_2)+\omega_{Y_0}(2-s_2))du_1du_2\Big\}\ |\ X_0,Y_0
\Big].
\end{aligned}
\end{equation}
The effective variance is then
\begin{equation}\label{e.defnu}
\nu_{\mathrm{eff}}^2=\int_{\R^{2d}}\int_{[0,1]^2}\E_\pi[\mathscr{H}_{\infty,\infty}(\tilde{X},\tilde{Y},x_1,x_2,s_1,s_2)]
\prod_{i=1}^2 \phi(s_i)\psi(x_i)ds_1ds_2dx_1dx_2,
\end{equation}
with $\tilde{X}$ and  $\tilde{Y}$ 
sampled, independently, from the invariant measure of $\hat{\pi}$.
 
In the following, we fix $x_1,x_2\in\R^d$ and $s_1,s_2\in[0,1]$, and simply write $\mathscr{H}_{M_1,M_2}(X_0,Y_0)$. The next two lemmas show the convergence
\begin{equation}\label{sep2002}
\cF_\eps(r,y,M_1,M_2) \to \nu_{\mathrm{eff}}^2  |\bar{g}(t-r,y)\bar{u}(r,y)|^2,~~\hbox{as $\eps\to0$ and $M_1,M_2\to\infty$, }
\end{equation}
for fixed $r\in(0,t),y\in\R^d$. 
\begin{lemma}\label{l.pointbd}
There exists $C>0$ independent of $\eps,M_1,M_2,x_1,x_2,s_1,s_2$ such that 
\begin{equation}\label{sep2006}
\what{\E}_{B,{t}/{\eps^2}}[\I_\eps e^{\J_\eps(M_1,M_2)}]\leq C.
\end{equation}
\end{lemma}
\begin{lemma}\label{l.pointcon}
As $\eps\to0$ and $M_1,M_2\to\infty$, we have
\[
\what{\E}_{B,{t}/{\eps^2}}[\I_\eps e^{\J_\eps(M_1,M_2)}]\to 
\E_\pi[\mathscr{H}_{\infty,\infty}(\tilde{X},\tilde{Y})]|\bar{g}(t-r,y)\bar{u}(r,y)|^2.
\]
\end{lemma}
  
\begin{proof}[Proof of Lemma~\ref{l.pointbd}]
Since $\I_\eps$ and $\cG_\eps$ are both bounded,  we have
\[
\what{\E}_{B,{t}/{\eps^2}}[\I_\eps e^{\J_\eps}] \les \E_{\pi}[e^{\J_\eps}].
\]
We first condition on $\tilde{X}_\eps,\tilde{Y}_\eps$ and assume that
\[
\frac{t-r}{\eps^2}+M_i\leq \tau+N_\eps.
\]
In this case, $\J_\eps$ is not related to $X_{N_\eps+1}, Y_{N_\eps+1}$ (which are sampled differently), and we can replace~$B^1,B^2$ 
with $\omega_{\tilde{X}_\eps},\omega_{\tilde{Y}_\eps}$, that is, the homogeneous chains started 
from $\tilde{X}_\eps,\tilde{Y}_\eps$, respectively, with the transition kernel $\hat{\pi}$. It is easy to check that in this case
\begin{equation}\label{sep2004}
 \E_{\pi}[e^{\J_\eps(M_1,M_2)} | \tilde{X}_\eps,\tilde{Y}_\eps]=\mathscr{H}_{M_1,M_2}(\tilde{X}_\eps,\tilde{Y}_\eps).
\end{equation}
In the case when $\J_\eps$ involves the last increment $X_{N_\eps+1},Y_{N_\eps+1}$, 
it is clear that we still have (\ref{sep2004}),  with
equality replaced by $\les$. 
%\sout{in the above expression.}

By Corollary~\ref{c.intersection}, we have
\[
\begin{aligned}
%&\E_\pi[e^{\lambda^2\int_{-1}^{T_1}\int_{-1}^{T_2}R_\phi(u_1,u_2)R_\psi(x_1-x_2+\omega_{\tilde{X}_\eps}(2+u_1)-\omega_{\tilde{X}_\eps}(2-s_1)-\omega_{\tilde{Y}_\eps}(2+u_2)+\omega_{\tilde{Y}_\eps}(2-s_2))du_1du_2}\ | \ \tilde{X}_\eps,\tilde{Y}_\eps] \\
\mathscr{H}_{M_1,M_2}(\tilde{X}_\eps,\tilde{Y}_\eps)\les1,
\end{aligned}
 \]
uniformly in $x_1,x_2\in\R^d,s_1,s_2\in[0,1],M_1,M_2>0$ and $\tilde{X}_\eps,\tilde{Y}_\eps\in\Omega$, and (\ref{sep2006}) follows.
\end{proof}
  
\begin{proof}[Proof of Lemma~\ref{l.pointcon}]
We divide the proof into three steps. 

\emph{Step 1}. We claim that for any $\delta>0$, there exists a universal $M>0$ such that if $\min(M_1,M_2)>M$, we have
\begin{equation}\label{e.cutoff}
\what{\E}_{B,{t}/{\eps^2}}\Big[|e^{\J_\eps(M_1,M_2)}-e^{\J_\eps(M,M)}|\Big]<\delta.
\end{equation}
First, since $R_\phi,R_\psi\geq0$,  $\cG_\eps$ is bounded 
and $R_\phi(u_1,u_2)$ is supported on $|u_1-u_2|\le 1$, we have
\begin{equation}\label{e.e201}
\begin{aligned}
\what{\E}_{B,{t}/{\eps^2}}\Big[|e^{\J_\eps(M_1,M_2)}-e^{\J_\eps(M,M)}|\Big] 
\les \E_{\pi}\Big[ e^{\J_\eps(M_1,M_2)}1_{\{\mathcal{E}_1(M)>0\}}\Big],
\end{aligned}
\end{equation}
with 
\[
\mathcal{E}_1(M)=\sup_{M_1,M_2>M}\int_{M-1}^{M_1}\int_{M-1}^{M_2}R_\phi(u_1,u_2)
R_\psi(x_1-x_2+B_{\frac{t-r}{\eps^2}+u_1}^1-B_{\frac{t-r}{\eps^2}-s_1}^1
-B_{\frac{t-r}{\eps^2}+u_2}^2+B_{\frac{t-r}{\eps^2}-s_2}^2)du_1du_2.
\]
After applying the Cauchy-Schwarz inequality to the r.h.s. of \eqref{e.e201} and using Lemma~\ref{l.pointbd}, 
we only need to consider $\pi[\mathcal{E}_1(M)>0]$, which is essentially the same as the probability of the ``nearby 
time'' of~$B^1,B^2$ being greater than $M$. By the same argument  as in
the proof of Lemma~\ref{l.pointbd}, Proposition~\ref{p.expintegrability} and Corollary~\ref{c.intersection}, 
we have $\pi[\mathcal{E}_1]\to0$ as $M\to\infty$, which proves \eqref{e.cutoff}.
  
\emph{Step 2}. We show that  
\begin{equation}\label{sep2012}
\what{\E}_{B,{t}/{\eps^2}}[(\I_\eps-\tilde{\I}_\eps)e^{\J_\eps(M,M)}]\to0,~~\hbox{as $\eps\to0$,}
\end{equation} 
where
\[
\tilde{\I}_\eps=\prod_{i=1}^2 g(y-\eps B_{{(t-r)}/{\eps^2}-\eps^{-\alpha}}^i) 
u_0(y+\eps B_{{t}/{\eps^2}-\eps^{-\alpha}}^i-\eps B_{T_M^\eps}^i),
\]
with
\[
T_M^\eps=\min\{T_i:  \frac{t-r}{\eps^2}+\eps^{-\alpha}\leq T_i\leq \frac{t}{\eps^2}-\eps^{-\alpha}\},
\]
and the  convention that   $T_M^\eps={t}/{\eps^2}-\eps^{-\alpha}$ 
if there is no regeneration time in 
the interval $[{(t-r)}/{\eps^2}+\eps^{-\alpha},{t}/{\eps^2}-\eps^{-\alpha}]$. 
As 
$\cG_\eps$ and $\J_\eps(M,M)$ are bounded, we have 
\begin{equation}\label{e.e202}
\begin{aligned}
\what{\E}_{B,{t}/{\eps^2}}[|\I_\eps-\tilde{\I}_\eps|e^{\J_\eps(M,M)}] \les \E_\pi[|\I_\eps-\tilde{\I}_\eps|].
\end{aligned}
\end{equation}
By Proposition~\ref{p.clt}, we have the convergence in distribution of
\begin{equation}\label{e.wkcon}
\begin{aligned}
(\eps B_{\frac{t-r}{\eps^2}-\eps^{-\alpha}}^i, \eps B_{\frac{t-r}{\eps^2}-s_i}^i, \eps B_{T_M^\eps}^i, 
\eps B_{\frac{t}{\eps^2}-\eps^{-\alpha}}^i, \eps B_{\frac{t}{\eps^2}}^i)\Rightarrow (W_{t-r}^i,W_{t-r}^i,W_{t-r}^i,W_t^i, W_t^i),
\end{aligned}
\end{equation}
which implies that
the r.h.s. of \eqref{e.e202} goes to zero as $\eps\to0$.
  
\emph{Step 3}. We prove the convergence of 
\begin{equation}\label{e.constep3}
\what{\E}_{B,{t}/{\eps^2}}[\tilde{\I}_\eps e^{\J_\eps(M,M)}]\to 
\E_{\pi}[\mathscr{H}_{M,M}(\tilde{X},\tilde{Y})]|\bar{g}(t-r,y)\bar{u}(r,y)|^2,
\end{equation}
where $\tilde{X},\tilde{Y}$ are sampled independently from the invariant measure of $\hat{\pi}$. 
%\sout{We choose $\eps\ll1$ so that~$M<\eps^{-\alpha}$.} 
First, we have
\[
\what{\E}_{B,{t}/{\eps^2}}[\tilde{\I}_\eps e^{\J_\eps(M,M)}]=
\E_\pi[\tilde{\I}_\eps e^{\J_\eps(M,M)} \cG_\eps(X_{N_\eps})\cG_\eps(Y_{N_\eps})].
\]
Note that, for $\eps$ sufficiently small (depending on $M$ and $r$),
both $\tilde{\I}_\eps$ and ${\J_\eps(M,M)}$ depend only on~$\{B^i_s: s\leq {t}/{\eps^2}-\eps^{-\alpha}\}$. 
Lemma~\ref{l.rnde} implies that it suffices to prove the convergence of 
$\E_\pi[\tilde{\I}_\eps \exp\{\J_\eps(M,M)\} ]$. We write
\begin{equation}\label{e.e222}
\E_\pi[\tilde{\I}_\eps e^{\J_\eps(M,M)}]
=\E_\pi\Big[\prod_{i=1}^2g(y-\eps B_{{(t-r)}/{\eps^2}-\eps^{-\alpha}}^i)\mathscr{H}_{M,M}(\tilde{X}_\eps,\tilde{Y}_\eps)\Big]
\E_\pi\Big[\prod_{i=1}^2u_0(y+\eps B_{{t}/{\eps^2}-\eps^{-\alpha}}^i-\eps B_{T_M^\eps}^i)\Big].
\end{equation}
Here, we used the independence of the increments 
after the regeneration time $T_M^\eps$ to split off the second factor, and
the separation between the time $(t-r)/\eps^2-\eps^{-\alpha}$ 
and the times appearing in the integration in $\J_\eps(M,M)$
in the first factor. 
By the weak convergence in \eqref{e.wkcon}, we have
 \[
\E_\pi\Big[\prod_{i=1}^2u_0(y+\eps B_{{t}/{\eps^2}-\eps^{-\alpha}}^i-\eps B_{T_M^\eps}^i)\Big]\to |\bar{u}(r,y)|^2.
\]
It remains to consider the first factor in the right side
of \eqref{e.e222}. We claim that as $\eps\to0$
\begin{equation}\label{e.coninvariant}
\E_\pi\Big[\prod_{i=1}^2g(y-\eps B_{\frac{t-r}{\eps^2}-\eps^{-\alpha}}^i)\mathscr{H}_{M,M}(\tilde{X}_\eps,\tilde{Y}_\eps)\Big]
-\E_\pi\Big[\prod_{i=1}^2g(y-\eps B_{\frac{t-r}{\eps^2}-\eps^{-\alpha}}^i)\Big]
\E_{\pi}\Big[\mathscr{H}_{M,M}(\tilde{X},\tilde{Y})\Big]\to0.
\end{equation}
The proof of \eqref{e.coninvariant} is the same as the proof of Lemma~\ref{l.rnde}, as 
$\mathscr{H}_{M,M}$ is bounded and $\tilde{X}_\eps,\tilde{Y}_\eps$ are the increments of $B^1,B^2$ on 
the interval~$[(t-r)/{\eps^2}-2,{(t-r)}/{\eps^2}-1]$. 
We apply the weak convergence~\eqref{e.wkcon} again to get
 \[
\E_\pi\Big[\prod_{i=1}^2g(y-\eps B_{{(t-r)}/{\eps^2}-\eps^{-\alpha}}^i)\Big]\to |\bar{g}(t-r,y)|^2
\]
and complete the proof of \eqref{e.constep3}.
  
Combining steps 1-3 and sending $\delta\to0$, completes the proof.
 \end{proof}

 \subsection{Proof of the uniform bound (\ref{e.bdf})}

We now prove the uniform bound (\ref{e.bdf}) in Proposition~\ref{p.varcon}.  
By Lemma~\ref{l.pointbd}, we have
\[
\what{\E}_{B,{t}/{\eps^2}}[|e^{\J_\eps(M_1,M_2)}|^2]\les 1,
\]
so by the Cauchy-Schwarz inequality, 
\[
| \cF_\eps(r,y,M_1,M_2)| \les \int_{\R^{d}}\int_{[0,1]} \what{\E}_{B,{t}/{\eps^2}}
[|g(\eps x+y-\eps B_{{(t-r)}/{\eps^2}-s})|] \phi(s)\psi(x)dsdx.
\]

\begin{lemma}\label{l.tailestimate}
For any $k\in\Z_{\geq1}$, there exists $C_k$ such that 
\begin{equation}
\what{\E}_{B,{t}/{\eps^2}}[1_{\{|\eps B_{{(t-r)}/{\eps^2}-s}|>M\}}] \leq \frac{C_k}{M^{2k}}
\end{equation}
for all $M>0$.
\end{lemma}

By the above lemma and the fact that $g\in\C_c^\infty(\R^d), |x|\leq 1$, we have 
\[
\what{\E}_{B,{t}/{\eps^2}}[|g(\eps x+y-\eps B_{{(t-r)}/{\eps^2}-s})|] \les 1\wedge  \frac{1}{|y|^k},
\]
which implies \eqref{e.bdf} and finishes the proof of Proposition~\ref{p.varcon}.

\begin{proof}[Proof of Lemma~\ref{l.tailestimate}]
Since $\cG_\eps$ is bounded, it suffices to prove the same estimate for 
\[
\pi[|\eps B_{{(t-r)}/{\eps^2}-s}|>M].
\]
We will assume $r=0,s=0$ to simplify the notation and the proof of the general case is the same. First, we 
write $B_{t/\eps^2}$ as a
sum of independent zero-mean random variables using the regeneration structure. 
Let  $\tau=1$ and $N_\eps=[t/\eps^2]-1$ and set
\[
B_{t/\eps^2}=\sum_{k=0}^{N_\eps}X_k(1)+B_{{t}/{\eps^2}}-B_{[{t}/{\eps^2}]}.
\]
We also write  
\[
B_{{t}/{\eps^2}}=\sum_{j=0}^{K_\eps} \X_j,
\]
where 
\[
\X_j=\left\{
\begin{array}{ll}
\sum_{k=T_j}^{T_{j+1}-1}X_k(1), & j=0,\ldots,K_\eps-1,\\
B_{{t}/{\eps^2}}-B_{T_{K_\eps}}, & j=K_\eps,
\end{array}
\right.
\] 
and $K_\eps=\max\{j: T_j\leq N_\eps\}$.
%\sout{Now we freeze $\{T_j\}_{j=0}^{K_\eps}$. } 
Since $\X_j$ are independent random variables with zero mean conditioning on $\{T_j\}_{j=0}^{K_\eps}$, 
the sum
%\sout{we can define and the following martingale }
\[
\M_k=\sum_{j=0}^k \X_j,  \   \  k=0,\ldots,K_\eps,
\]
is a martingale.  
By the Chebyshev and martingale inequalities, we have 
\[
\pi\Big[|\eps \M_{K_\eps}|>M \ | \ \{T_j\}_{j=0}^{K_\eps}\Big] \leq \frac{1}{M^{2k}}
\E_\pi\Big[|\eps \M_{K_\eps}|^{2k} \ |\ \{T_j\}_{j=0}^{K_\eps}\Big]
\les  \frac{1}{M^{2k}} \E_\pi\Big[  \Big|\eps^2\sum_{j=0}^{K_\eps} \X_j^2\Big|^k  \ | \ \{T_j\}_{j=0}^{K_\eps}\Big].
\]
Since $K_\eps\leq N_\eps$, we only need to show that
\begin{equation}\label{e.e223}
\eps^{2k}\E_\pi\Big[\Big|\sum_{j=0}^{N_\eps} \X_j^2\Big|^k\Big]\les1.
\end{equation}
If we expand $|\sum_{j=0}^{N_\eps} \X_j^2|^k$, the number of terms is smaller than $(t/\eps^2)^k$, and each 
term is of the form $\prod_{l=1}^k \X_{j_l}^2$ for some $ j_l=0,\ldots,N_\eps$, whose expectation 
is uniformly bounded, in light of Lemma~\ref{l.meanvar}. Thus, \eqref{e.e223} holds and the proof is complete.
\end{proof}

\appendix 

\section{Technical lemmas}\label{sec:append-tech}

%\subsection{Proof of Lemma \ref{l.renormalizationconstant}}
\begin{lemma}\label{l.renormalizationconstant}
There exist $c_1,c_2$ such that 
\begin{equation}\label{e.zetatc1c2}
\zeta_T:=\log\E_B\Big[e^{\frac12\lambda^2\int_{[0,T]^2}R(s-u,B_s-B_u)dsdu}\Big]=c_1T+c_2+o(1),~~\hbox{as $T\to\infty$.}
\end{equation}
\end{lemma}
\begin{proof}
%Recall that we need to prove that
%\begin{equation}\label{e.zetaT}
%\zeta_T=c_1T+c_2+o(1),~\hbox{as $T\to+\infty$}.
%\end{equation}
 We employ the setup of Section \ref{s.chain}.
The proof is divided into three steps, in which we prove 
that~\eqref{e.zetatc1c2} holds for $T\in \N,\mathbb{Q},\R$.

\emph{Step 1}, $T\in \N$. In the construction of the chain,  
set $\tau=1$. 
As in Section~\ref{s.constructionchain}, we have
\[
\what{\Pb}_T(d\omega)=\Psi(x_0)\pi(dx_0)\left(\prod_{k=0}^{T-2} \hat{\pi}(x_k,x_{k+1})\right)\Psi^{-1}(x_{T-1}) \rho^{T-1}e^{T\zeta_1-\zeta_T}.
\]
Using the normalization (\ref{sep1808}) gives, 
\[
\E_\pi[\Psi^{-1}(X_{T-1})]=e^{\zeta_T-T\zeta_1}\rho^{1-T}.
\]
By \eqref{e.conexpfast}, we have 
\[
e^{\zeta_T-T\zeta_1}\rho^{1-T}=e^{\zeta_T-T\zeta_1-(T-1)\log \rho}\to \E_\pi[\Psi^{-1}(\tilde{X})]
\]
exponentially fast as $T\to\infty$, where $\tilde{X}$ is sampled from the invariant measure of $\hat{\pi}$. This  proves~\eqref{e.zetatc1c2} for integer $T$ with 
\begin{equation}\label{sep2206}
c_1=\zeta_1+\log \rho,  \   \ c_2=\log\rho^{-1}+
\log\E_\pi[\Psi^{-1}(\tilde{X})].
\end{equation}
We note that the convergence rate of the remainder $o(1)\to0$ as $T\to\infty$ 
only depends on the estimates on $\Psi$ and $\gamma$ which 
are determined by $\|I\|_{L^\infty}$.

\emph{Step 2}, $T\in\mathbb{Q}$. In the construction of the chain, the choice 
of the length-one increment is arbitrary  -- we can 
take any length that is greater than one and follow the same 
construction. Take the increment of length $r\in \mathbb{Q}$ such 
that $r\in(1,2)$ (then the corresponding $I(x,y)$ is uniformly bounded), 
so there exist $m_1,m_2\in \N$ such that $rm_1=m_2$. For any $k\in \N$, 
the same proof as in Step~1 
shows that 
\[
\zeta_{rm_1k}=c_{1,r}m_1k+c_{2,r}+o(1)
\]
for some $c_{1,r},c_{2,r}$. Since 
\[
\zeta_{m_2k}=c_1m_2k+c_2+o(1)
\]
from step 1, we 
conclude that
$c_{1,r}=c_1r$ and $c_{2,r}=c_2$ by sending $k\to\infty$. Thus, for any $r\in\mathbb{Q}$, we have 
\[
\zeta_{rk}=c_1rk+c_2+o(1),
\]
with $o(1)\to0$ as $k\to\infty$, uniformly in $r\in(1,2)$. Choosing $r=T/[T]$, 
we see that \eqref{e.zetatc1c2} holds for~$T\in\mathbb{Q}$.

\emph{Step 3}, $T\in\R$. 
As $\zeta_T$ is continuous in $T$, we simply
take $T_n\in\mathbb{Q}$ so that $T_n\to T$ and~$\zeta_{T_n}\to \zeta_T$. Since 
\[
\zeta_{T_n}=c_1T_n+c_2+o(1)
\]
with $o(1)\to0$ as $T_n\to\infty$, the proof is complete.
\end{proof}

\begin{lemma}\label{l.meanvar}
Assuming $X_0\sim \pi(dx_0)$, $X_{k+1}\sim 
(1-\gamma)^{-1}(\hat{\pi}(X_{k},dx_{k+1})-\gamma\pi(dx_{k+1}))$ 
for $k\geq 0$, and~$\theta$ is an independent geometric random variable 
with parameter $\gamma$. Then, for all $k\geq 0$, $\E_\pi[X_k(1)]=0$ and 
there exists~$c>0$ such that 
\begin{equation}\label{e.gausstail}
\pi\big[\max_{s\in[0,1]}|X_k(s)|\geq t\big]\les e^{-ct^2}, 
\end{equation}
\begin{equation}\label{e.exptail}
 \pi\Big[\sum_{k=0}^\theta \max_{s\in[0,1]}|X_k(s)|>t\Big]\les e^{-ct}.
\end{equation}
\end{lemma}

\begin{proof}
For any measure $\nu_0$ on $\Omega$ that is symmetric, so
that 
 $\nu_0(A)=\nu_0(-A)$ with 
 $-A:=\{f: -f \in A\}$, set 
\[
\nu_1(A)=\int_{\Omega}\nu_0(dx)\hat{\pi}(x,A).
\]
Recall that 
\[
\int_{\Omega} e^{I(x,y)}\Psi(y)\pi(dy)=\rho \Psi(x).
\] 
%\sout{and $I$ is defined in \eqref{e.defI}. By the fact that } 
Since 
$I(x,y)=I(-x,-y)$, $\pi$ is symmetric, and $\Psi$ is the unique eigenvector
corresponding to
$\rho$ satisfying (\ref{sep1808}),
%\sout{$\int_\Omega \Psi(y)\pi(dy)=1$, }
we have {that} $\Psi(-x)=\Psi(x)$, hence
$\hat{\pi}(x,A)=\hat{\pi}(-x,-A)$ and  $\nu_1$ is symmetric.
Thus, the distribution of $X_k$ is symmetric, and
\[
\E_\pi[X_k(1)]=-\E_\pi[X_k(1)]=0.
\]

For the Gaussian tail in \eqref{e.gausstail}, we note that
\begin{equation}\label{e.bdpro}
\sup_{x\in\Omega} \frac{\hat{\pi}(x,dy)-\gamma\pi(dy)}{1-\gamma}\les \sup_{x\in\Omega}\hat{\pi}(x,dy)\les \pi(dy).
\end{equation}
As $\pi$ is the Wiener measure on $\C([0,1])$ tilted by the bounded factor 
\[
\exp\Big\{\frac12\lambda^2\int_{[0,1]^2}R(s-u,\omega(s)-\omega(u))dsdu-\zeta_1
\Big\},
\] 
there exists $c>0$ such that 
\begin{equation}\label{sep2208}
\pi[\max_{s\in[0,1]}|X_{k+1}(s)|\geq t\ | \ X_{k}] \les \pi[\max_{s\in[0,1]}|X_0(s)|\geq t] \les e^{-ct^2}
\end{equation}
uniformly in $X_{k}$. After averaging with respect to $X_k$, we obtain
 \eqref{e.gausstail}.

To prove \eqref{e.exptail}, we note that
\[
\pi[\theta>[\alpha t]]\les (1-\gamma)^{\alpha t}
\]
for any $\alpha>0$. By the Chebyshev inequality, we have
\[
\pi\Big[\sum_{k=0}^{[\alpha t]}\max_{s\in[0,1]}|X_k(s)| >t\Big]
\leq e^{-C_1t}\E_\pi\Big[
\exp\Big\{C_1\sum_{k=0}^{[\alpha t]}\max_{s\in[0,1]}|X_k(s)| \Big\}\Big]
\]
for any $C_1>0$. Using \eqref{e.bdpro} again, we have 
\[
e^{-C_1t}\E\Big[\exp\Big\{C_1\sum_{k=0}^{[\alpha t]}\max_{s\in[0,1]}|X_k(s)| 
\Big\}\Big]\les e^{-C_1t} C_2^{[\alpha t]}
\]
for some constant $C_2>0$ independent of $\alpha$. 
Taking  $\alpha<{C_1}/{\log C_2}$ finishes
the proof.  \end{proof}

\begin{lemma}\label{l.rnde}
If  
$F:\Omega_{t/\eps^2}\to\R$ is bounded and only depends on $X_0,\ldots,X_{M_\eps}$, with $N_\eps-M_\eps\to\infty$, then 
\begin{equation}\label{e.rnde1}
| \E_\pi[F(B)\cG_\eps(X_{N_\eps})]-\E_\pi[F(B)]|\to0
\end{equation}
as $\eps\to0$.
\end{lemma}
\begin{proof}
First, we have 
\[
 | \E_\pi[F(B)\cG_\eps(X_{N_\eps})]-\E_\pi[F(B)]|=|\E_\pi[F(B)\E_\pi[\cG_\eps(X_{N_\eps})-1| X_{M_\eps}]]|
  \les \E_\pi[|\E_\pi[\cG_\eps(X_{N_\eps})| X_{M_\eps}]-1|].
\]
  Since $\cG_\eps$ is bounded, by \eqref{e.conexpfast}, we have 
  \[
  |\E_\pi[\cG_\eps(X_{N_\eps})| X_{M_\eps}]-\E_{\pi}[\cG_\eps(\tilde{X})]|\to0 \hbox{ as $\eps\to 0 $},
  \]
  uniformly in $X_{M_\eps}$. Here, $\tilde{X}$ is sampled from the invariant measure of $\hat{\pi}$. Since
  $\E_{\pi}[\cG_{\eps}(X_{N_\eps})]=1$, we know 
  that $\E_{\pi}[\cG_\eps(\tilde{X})]\to1$ as $\eps\to 0 $.
  Hence,
  \[
    \E_\pi[\cG_\eps(X_{N_\eps})| X_{M_\eps}]-1\to 0, \hbox{ as $\eps\to 0 $},
  \]
which completes the proof.
 \end{proof}

\begin{lemma}\label{l.bdsum4}
There exists $C>0$ independent of $\eps$ such that $\sum_j \sqrt{\E[|\cX_j^\eps|^4]}\leq C$.
\end{lemma}

\begin{proof}
 Recall that 
\[
\cX_j^\eps=\frac{1}{\eps^{{d}/{2}-1}} \int_{I_{\beta,j}}\int_{\R^d}\tilde{Z}_t^\eps(r,y)dW(r,y),
\]
and by the martingale inequality, we have
\[
\E[|\cX_j^\eps|^4] \les \frac{1}{\eps^{2d-4}}\int_{I_{\beta,j}^2}\int_{\R^{2d}} \E[|\tilde{Z}_t^\eps(r,y)\tilde{Z}_t^\eps(r',y')|^2]dydy'drdr'.
\]
For $\E[|\tilde{Z}_t^\eps(r,y)\tilde{Z}_t^\eps(r',y')|^2]$, we repeat the calculation in the proof of Lemma~\ref{l.varre}. To simplify the notation, we let $r=r_1=r_2,r'=r_3=r_4$ and $y=y_1=y_2,y'=y_3=y_4$ and consider 
\[
\E\Big[\prod_{i=1}^4 \tilde{Z}_t^\eps(r_i,y_i)\Big]=\!
\int_{\R^{4d}}  \E\what{\E}_{B,{t}/{\eps^2}}\Big[\prod_{i=1}^4g(x_i) u_0(x_i+\eps B_{{t}/{\eps^2}}^i)\Phi^\eps_{t,x_i,B^i}(r_i,y_i)e^{\lambda \tilde{M}^\eps_{t,x_i,B^i}(r_i)-\frac12\lambda^2\la \tilde{M}^\eps_{t,x_i,B^i}\ra_{r_i}} \Big] d\bfx.
\]
As in the proof of Lemma~\ref{l.varre}, we obtain 
%\gucomment{the expression for the identity is too complicated to write down, so I use $\leq$ instead.}
\[
\begin{aligned}
&\frac{1}{\eps^{2d-4}}\int_{I_{\beta,j}^2}\int_{\R^{2d}} \E[|\tilde{Z}_t^\eps(r,y)\tilde{Z}_t^\eps(r',y')|^2]dydy'drdr'\\
&\leq \int_{[0,t]^2}\int_{\R^{6d}}\int_{[0,1]^4} 1_{\{{r}/{\eps^2}\in I_{\beta,j}\}}1_{\{{r'}/{\eps^2}\in I_{\beta,j}\}}
\what{\E}_{B,{t}/{\eps^2}}[Ie^{J}]\prod_{i=1}^4\phi(s_i)\psi(x_i) d\bfs d\bfx dy dy' dr dr',
\end{aligned}
\]
where 
\[
I=\prod_{i=1}^4 |g(\eps x_i+y_i-\eps B_{{(t-r_i)}/{\eps^2}-s_i}^i)u_0(\eps x_i+y_i+\eps B_{{t}/{\eps^2}}^i-\eps B_{{(t-r_i)}/{\eps^2}-s_i}^i)|,
\]
\[
\begin{aligned}
J=&\lambda^2\sum_{1\leq i<l\leq 4}\int_{-1}^{1/2\eps^\alpha}\int_{-1}^{1/2\eps^\alpha}R_\phi(u_i,u_l)\\
&\times R_\psi(x_i-x_l+\frac{y_i-y_l}{\eps}+B_{\frac{t-r_i\wedge r_l}{\eps^2}+u_i}^i-B_{\frac{t-r_i}{\eps^2}-s_i}^i-B_{\frac{t-r_i\wedge r_l}{\eps^2}+u_l}^l+B_{\frac{t-r_l}{\eps^2}-s_l}^l)du_idu_l.
\end{aligned}
\]
By the same proof  {as that of}  \eqref{e.bdf}, we have 
\[
\begin{aligned}
\frac{1}{\eps^{2d-4}}\int_{I_{\beta,j}^2}\int_{\R^{2d}} \E[|\tilde{Z}_t^\eps(r,y)\tilde{Z}_t^\eps(r',y')|^2]dydy'drdr'\les \left(\int_0^t 1_{\{{r}/{\eps^2}\in I_{\beta,j}\}}dr\right)^2.
\end{aligned}
\]
The proof is complete.
\end{proof}


\begin{thebibliography}{10}

\bibitem{alberts2014intermediate}
{\sc T.~Alberts, K.~Khanin, J.~Quastel}, {\em The intermediate disorder
  regime for directed polymers in dimension $1+ 1$}, Ann. Probab.,
  {\bf 42} (2014), pp.~1212--1256.

\bibitem{amir2011probability}
{\sc G.~Amir, I.~Corwin, and J.~Quastel}, {\em Probability distribution of the
  free energy of the continuum directed random polymer in 1+ 1 dimensions},
  Comm. Pure Appl. Math., {\bf 64} (2011), pp.~466--537.

\bibitem{bertini1997stochastic}
  {\sc L.~Bertini and G.~Giacomin}, {\em Stochastic {Burgers} and {KPZ}
  equations
  from particle systems}, Comm. Math. Phys., {\bf 183} (1997),
  pp.~571--607.

\bibitem{betz2005central}
  {\sc V.~Betz and H.~Spohn}, {\em A central limit theorem for {Gibbs}
  measures
  relative to {Brownian} motion}, Prob. Theory Rel. Fields, {\bf 131}
  (2005), pp.~459--478.

\bibitem{billingsley}
  {\sc P.~Billingsley}, Convergence of probability
  measures, Academic Press (1999).

\bibitem{bolthausen1989note}
{\sc E.~Bolthausen}, {\em A note on the diffusion of directed polymers in a
  random environment}, Comm. Math. Phys., {\bf 123} (1989),
  pp.~529--534.

\bibitem{caravenna2015universality}
{\sc F.~Caravenna, R.~Sun, and N.~Zygouras}, {\em Universality in marginally
  relevant disordered systems},  Ann. Appl. Prob., {\bf 27} (2017), pp.~3050--3112.
    
  \bibitem{chandra2017moment}
{\sc A.~Chandra and H.~Shen}, {\em Moment bounds for SPDEs with non-Gaussian
  fields and application to the Wong-Zakai problem}, Electron. J. Probab., {\bf 22} (2017), paper no. 68.
  
  \bibitem{comets2}
  {\sc F.~Comets}, {\em Directed polymers in random environments}, Lecture Notes in Mathematics,
vol. 2175, Springer, Cham, 2017, Lecture notes from the 46th Probability Summer School
held in Saint-Flour, 2016.
  
  \bibitem{comets1}
  {\sc F.~Comets and Q.~Liu}, {\em Rate of convergence for polymers in a weak disorder}, 
  J. Math. Anal. Appl.,  {\bf 455} (2017), pp.~312--335.


\bibitem{dautraylionsvol3}
{\sc R. Dautray and J,-L. Lions}, {\em Mathematical Analysis and Numerical Methods in Science and Technology.
Volume 3: Spectral Theory and Applicaitons}, Springer, 1990. 

\bibitem{feng2016rescaled}
{\sc Z.~S. Feng}, {\em Rescaled Directed Random Polymer in Random Environment
  in Dimension 1+ 2}, PhD thesis, University of Toronto (Canada), 2016.

\bibitem{gu2015homogenization}
{\sc Y.~Gu and G.~Bal}, {\em Homogenization of parabolic equations with large
  time-dependent random potential}, Stoch. Proc. Appl., {\bf 125} (2015), pp.~91--115.
  
  \bibitem{gu2018another}
{\sc Y.~Gu and L.-C. Tsai}, 
{\em Another look into the Wong-Zakai theorem for stochastic heat
  equation}, 
{\em arXiv preprint arXiv:1801.09164}, 2018.

\bibitem{gubinelli2006gibbs}
  {\sc M.~Gubinelli}, {\em Gibbs measures for self-interacting {Wiener} paths},
  Markov Proc. Rel. Fields, {\bf 12} (2006), pp.~747--766.

\bibitem{gubinelli2015paracontrolled}
{\sc M.~Gubinelli, P.~Imkeller, and N.~Perkowski}, {\em Paracontrolled
  distributions and singular {PDEs}}, in Forum of Mathematics, Pi, vol.~3,
  Cambridge Univ Press, 2015, p.~e6.

\bibitem{hairer2013solving}
  {\sc M.~Hairer}, {\em Solving the {KPZ}
equation}, Ann. Math., {\bf 178}
  (2013), pp.~559--664.

\bibitem{hairer2014theory}
\leavevmode\vrule height 2pt depth -1.6pt width 23pt, {\em A theory of
  regularity structures}, Inv. Math., {\bf 198} (2014), pp.~269--504.
  
  \bibitem{hairer2015multiplicative}
{\sc M.~Hairer and C.~Labb{\'e}}, {\em Multiplicative stochastic heat equations
  on the whole space}, J. Eur. Math. Soc., {\bf 20} (2018), pp.~1005--1054.

\bibitem{hairer2015wong}
  {\sc M.~Hairer and {\'E}.~Pardoux}, {\em A {Wong-Zakai} theorem for stochastic
  {PDEs}}, Jour. Math. Soc. Japan, {\bf 67} (2015),
  pp.~1551--1604.
  
  \bibitem{hu2009ecp}
    {\sc Y.~Hu and D.~Nualart}, {\em  Stochastic integral representation of the $L^2$ modulus of continuity of Brownian local time and a central limit theorem},
  Electr. Comm. Prob.,{\bf 14} (2009), pp.~529--539.
  
  
%  \bibitem{hu2011aop}
%  {\sc Y.~Hu, D.~Nualart, and J.~Song}, {\em Feynman-Kac formula for heat equation driven by fractional white noise}, 
%  Ann. Probab., {\bf 39} (2011), pp.~291--326.
  
  \bibitem{hu2008integral}
{\sc Y.~Hu, D.~Nualart, and J.~Song}, {\em Integral representation of
  renormalized self-intersection local times}, 
Jour. Funct. Anal.,
  {\bf 255} (2008), pp.~2507--2532.

\bibitem{imbrie1988diffusion}
{\sc J.~Z. Imbrie and T.~Spencer}, {\em Diffusion of directed polymers in a
  random environment}, Jour. Stat. Phys., {\bf 52} (1988), pp.~609--626.

\bibitem{kupiainen2016renormalization}
{\sc A.~Kupiainen}, {\em Renormalization group and stochastic
  {PDEs}}, in Annales
  Henri Poincar{\'e}, vol.~17, Springer, 2016, pp.~497--535.

\bibitem{magnen2017diffusive}
{\sc J.~Magnen and J.~Unterberger}, {\em The Scaling Limit of the KPZ Equation in Space Dimension 3 and Higher}, 
Jour. Stat. Phys., {\bf 171} (2018), pp.~543--598.

\bibitem{mukherjee2017central}
{\sc C.~Mukherjee}, {\em A central limit theorem for the annealed path measures
  for the stochastic heat equation and the continuous directed polymer in $
  d\geq 3$}, arXiv preprint arXiv:1706.09345,  (2017).

\bibitem{mukherjee2016weak}
{\sc C.~Mukherjee, A.~Shamov, and O.~Zeitouni}, {\em Weak and strong disorder
  for the stochastic heat equation and continuous directed polymers in $ d\geq
  3$}, Electr. Comm. Prob., {\bf 21} (2016).
  
  \bibitem{otto2016quasilinear}
{\sc F.~Otto and H.~Weber}, {\em Quasilinear SPDEs via rough paths}, arXiv
  preprint arXiv:1605.09744,  (2016).

\bibitem{stone1965local}
{\sc C.~Stone}, {\em A local limit theorem for nonlattice multi-dimensional
  distribution functions},  Ann. Math. Stat., {\bf 36} (1965),
  pp.~546--551.

\end{thebibliography}
\end{document}